\documentclass[11pt,leqno]{amsart}
\usepackage{amsmath,amsthm, amsfonts,amssymb}
\usepackage{enumerate}
\usepackage{mathtools}
\usepackage{hyperref}
\usepackage{esint}
\usepackage[T1]{fontenc}
\usepackage{cite}
\usepackage{color}
\usepackage[english,polish]{babel}
\newcommand{\pl}[1]{\foreignlanguage{polish}{#1}}
\usepackage{tikz}

\usepackage{xmpmulti}
\usepackage{colortbl}
\usepackage{ulem}

\usepackage{xcolor}
\usepackage{pgfplots}

\newtheorem{thm}{Theorem}[section]
\newtheorem{lem}[thm]{Lemma}

\newtheorem{cor}[thm]{Corollary}
\newtheorem{pro}[thm]{Proposition}

\theoremstyle{definition}

\theoremstyle{remark}
\newtheorem*{remark}{Remark}
\newtheorem*{remark1}{Remark 1}
\newtheorem*{remark2}{Remark 2}

\numberwithin{equation}{section}

\newcommand\Cvp[1]{Cv_p(#1)}
\newcommand\Cvq[1]{Cv_q(#1)}
\newcommand{\RR}{\mathbb{R}}
\newcommand{\CC}{\mathbb{C}}
\newcommand{\DD}{\mathbb{D}}
\newcommand{\tm}{\tilde{m}}
\newcommand{\Nb}{\bar{N}}
\newcommand{\ZZ}{\mathbb{Z}}

\newcommand{\fg}{\frak g}
\newcommand{\fk}{\frak k}
\newcommand{\fp}{\frak p}
\newcommand{\fs}{\frak s}
\newcommand{\fl}{\frak l}

\newcommand{\nb}{\bar{n}}
\newcommand{\la}{\lambda}
\newcommand{\zns}{\zeta_{n,s}}

\newcommand{\bZ}{\mathbb{Z}}
\newcommand{\bR}{\mathbb{R}}
\newcommand{\ut}{u_{\theta}}
\newcommand{\uph}{u_{\varphi}}
\newcommand{\vp}{\varphi}
\newcommand{\ve}{\varepsilon}

\newcommand\bignorm[2]{\left.{\bigl\Vert{#1}\bigr\Vert_{#2}}\right.}

\newcommand{\mD}{\mathcal{D}}

\newcommand{\mP}{\mathcal{P}}
\newcommand{\mL}{\mathcal{L}}
\newcommand{\mR}{\mathcal{R}}
\newcommand{\mJ}{\mathcal{J}}
\newcommand{\mF}{\mathcal{F}}

\DeclareMathOperator{\Real}{Re}
\DeclareMathOperator{\Ima}{Im}

\DeclareMathOperator{\Dom}{Dom}
\DeclareMathOperator{\Ran}{Ran}

\DeclareMathOperator{\Par}{Par}
\DeclareMathOperator{\Int}{Int}

\providecommand{\flo}[1]{\left \lfloor #1 \right \rfloor }

\newcommand{\bN}{\mathbb{N}}

\newcommand{\TT}{\mathbb{T}}
\newcommand{\NN}{\mathbb{N}}

\DeclareFontEncoding{FMS}{}{}
\DeclareFontSubstitution{FMS}{futm}{m}{n}
\DeclareFontEncoding{FMX}{}{}
\DeclareFontSubstitution{FMX}{futm}{m}{n}
\DeclareSymbolFont{fouriersymbols}{FMS}{futm}{m}{n}
\DeclareSymbolFont{fourierlargesymbols}{FMX}{futm}{m}{n}
\DeclareMathDelimiter{\VERT}{\mathord}{fouriersymbols}{152}{fourierlargesymbols}{147}

\author{Fulvio Ricci}
\address{
	Fulvio Ricci\\
	Scuola Normale Superiore\\
	Piazza dei Cavalieri 7\\
	56126 Pisa, Italy}
\email{fricci@sns.it}

\author{B{\l}a{\.z}ej Wr{\'o}bel}
\address{ B{\l}a{\.z}ej Wr{\'o}bel\\
	Universit\"{a}t Bonn \\
	Mathematical Institute\\
	Endenicher Allee 60\\
	D--53115 Bonn \\
	Germany \&
	Instytut Matematyczny\\
	Uniwersytet \pl{Wroc{\lll}awski}\\
	Pl. Grun\-waldzki 2/4\\
	50-384 \pl{Wroc{\lll}aw}\\
	Poland}
\email{blazej.wrobel@math.uni.wroc.pl}



\title[]
{Spectral multipliers for functions of fixed $K$-type on $L^p(SL(2,\RR))$ 
}

\begin{document}
\selectlanguage{english} 	
\begin{abstract}
	We prove an $L^p$ spectral multiplier theorem for functions of the $K$-invariant sublaplacian $L$ acting on the space of functions of fixed $K$-type on the group $SL(2,\RR).$ As an application we compute the joint $L^p(SL(2,\RR))$ spectrum of $L$ and the derivative along $K$.
\end{abstract}
\maketitle 

\section{Introduction}

In  the Lie algebra $\fg=\fs\fl(2,\bR)$  of $G=SL(2,\bR)$  set
\begin{equation}
\label{eq:XYdef}
X=\begin{pmatrix} 0&1/2\\
-1/2&0 \end{pmatrix}\ ,\qquad Y_1=\begin{pmatrix} 0&1/2\\
1/2&0 \end{pmatrix}\ ,\qquad Y_2=\begin{pmatrix} 1/2&0\\
0&-1/2 \end{pmatrix}\ .
\end{equation} Then $X$ and $\{Y_1,Y_2\}$ generate, respectively, the two components $\fk$ and $\fp$ in the Cartan decomposition of $\fg$. The two left-invariant differential operators $X$ and $L=-Y_1^2-Y_2^2$ commute and generate the full algebra of left- and ${\rm Ad}(K)$-invariant (also called $K$-central) differential operators on $G$.

We denote by the same symbols the unique self-adjoint extensions of $-iX$ and $L$ to $L^2(G)$, which strongly commute, in the sense that all their spectral projections commute with each other.

The spectral decomposition of $L^p(G)$ relative to $-iX$ is into  $K$-types:
$$
L^p(G)=\sum_{n\in\frac12\bZ}V_n^p\ ,
$$
with
$$
V_n^p=\big\{f\in L^p(G):  f\big(g\exp(\theta X)\big)=e^{in\theta}f(g)\big\}\ .
$$
This furnishes an analogous decomposition of the multipliers of the sublaplacian
$$
m(L)=\sum_{n\in\bZ}T_n\mP_n\ ,
$$
where $m$ is a Borel function on $\RR,$ $\mP_n$ is the orthogonal projection of $L^2(G)$ on $V_n^2,$ while $T_n:V_n^2\longrightarrow V_n^2$ is
\begin{equation}\label{Tn}
T_n=m(L_{|_{V_n^2}})\ .
\end{equation}

In this paper we study $L^p$-boundedness of the operators $T_n$ in \eqref{Tn}. 

In the case $n=0$,  operators of the form $T_0$ have been widely studied in the literature, due to the identification $V^p_0\simeq L^p(G/K)$, which transforms $L_{|_{V^2_0}}$ into  the Laplace-Beltrami operator on the hyperbolic plane. We refer to the results of Stanton and Tomas \cite{StTo_1}, Anker \cite{A1},  Ionescu \cite{I2,I3}, and Meda and Vallarino \cite{MV}, providing, in the wider context of symmetric spaces,  conditions on $m$ which imply $L^p$-boundedness of $T_0$ for a given $p\in(1,\infty)$. These conditions on $m$ are of Mikhlin-H\"ormander type on the boundary of the region $\Delta_n$ (for $n=0$), to be defined in \eqref{Delta_n}, on which $m$ must be defined, bounded and holomorphic in the interior.

Our main result is a multiplier theorem of the same kind for general $n\in \frac12\bZ$.

Before describing the content of the paper in greater detail, we want to put this result in a broader perspective. We see the result of this paper a first step towards the analysis of joint multiplier operators $m(L,-iX)$ (or equivalently, convolution operators with $K$-central kernels). 

The topic of joint spectral multipliers for (strongly) commuting operators has achieved some attention over the last years in general contexts. The interested reader may consult the work by Albrecht \cite{Al}, Albrecht, Franks, and McIntosh \cite{AlFrMc}, Lancien, Lancien, and Le Merdy \cite{LanLanMer}, M\"uller, Ricci, and Stein  \cite{Mu:RiSt1,Mu:RiSt2}, Fraser \cite{Fr1,Fr2,Fr3}, Martini \cite{Martini_Phd,Martini_JFA,Martini_Annales}, and Sikora \cite{Sik}. 

Particularly relevant for the results of the present paper is the work of the second author \cite{WroJSMMSO} (see also \cite[Chapter 6]{PhD}), where joint multipliers of two operators were studied, with one of the two ($-iX$ in our case) having a $C^k$ functional calculus (finite order of smoothness of the multiplier function $m$ produces bounded operators $m(-iX)$ on $L^p(G)$), while the other operator ($L$ in our case) only having a holomorphic functional calculus (if $m(L)$ is bounded on $L^p(G),$ $p\neq 2,$ then $m$ admits a holomorphic extension outside of the $L^2$-spectrum). Theorem 3.1 in \cite{WroJSMMSO} provides a joint spectral multiplier theorem for the pair $(L,-iX)$. However, this result  is not satisfactory as it does not take into account the interaction between $L$ and $-iX$, which, although commuting, do not act on separate variables. In particular, the theorem in \cite{WroJSMMSO} only applies to multipliers which are defined on the product of the two individual spectra, and not just on the joint spectrum, which is strictly smaller.

The content of the paper is organized as follows. In Section \ref{sec:pre} we explain the notations used throughout the article. In Section \ref{sec:Kcent} we introduce the $K$-central spherical functions  $\zns$ - these are joint eigenfunctions of $(L,-iX)$ corresponding to the eigenvalues $(n,s(1-s))$ - and give  the $K$-central Plancherel and inversion formulas. Then in Section \ref{sec-estimates} we prove the main technical estimates for spherical functions. These are contained in Lemmas \ref{lem:spher_decom_small_t} and \ref{lem:spher_decom_large_t}. The estimates are then used to prove our main result, Theorem \ref{thm:multVnp}, which is a Mikhlin-H\"ormander multiplier theorem on $V_n^p$ for $1<p<\infty.$  Theorem \ref{thm:multVnp} implies continuous extendibility to $V_n^p$ of operators of the form $m(L_{|_{V_n^2}})$. The conditions on the holomorphic part of the multiplier are the same as in the Stanton-Tomas theorem for $n=0$. One also has to take into account the (finite) discrete part of the spectrum, which is absent when $n=0$.  We also obtain a holomorphic extension property of multipliers of $L,$ which are bounded on $V_n^p,$ see Proposition \ref{pro:CleSte}. This is in spirit of the results of Clerc and Stein \cite{CS}. Finally, using Theorem \ref{thm:multVnp} we describe the joint $L^p$-spectrum of the pair $(L,-iX),$ see Theorem \ref{thm:jointspectrum}. 

In view of applications to joint multiplier operators,  it is important to keep track of how constants involved in norm inequalities grow with $n$. Our proof of Theorem \ref{thm:multVnp} is based on splitting the kernel of $T_n$ into three parts: discrete, continuous-local, and continuous-global. We are able to prove estimates that grow polynomially in $n$ for both the discrete and the continuous-local part. Our estimates for the continuous-global part have a rapid growth that is controlled by $\Gamma(Cn^2)$.\footnote{A refinement of the splitting used to prove Theorem \ref{thm:multVnp} gives a bound for the continuous-global part which is uniform in $n$. This however worsens the polynomial bound for the continuous-local part to a control by $e^{Cn^2}$.} We believe that these estimates are far from being sharp, however better bounds seem to be unknown at this point. It would be interesting to understand if also the estimate for the continuous-global part can be made  polynomially growing, or even uniformly bounded in $n$. We hope to  to be able to return to this topic in the future.

\section{Preliminaries}
\label{sec:pre}
\subsection{General notation}
\quad

Throughout the paper by $\gamma$ we mean the function
\begin{equation}
\label{eq:defgamma}\gamma(s)=s(1-s),\qquad s\in \CC.
\end{equation} 

For $n\in \frac12\ZZ$ by $D_n$ we denote
$$D_n=\{s\in \frac12\bZ\colon s-|n|\in \bZ,\quad 1\leq s\leq |n|\}.$$ 
We write $\NN$ for the set of non-negative integers.  

For $t>0$ we let $S_t$ be the vertical strip  around $\Real(z)=1/2$ given by
$$S_{t}=\bigg\{z\in \CC\colon \frac12-t\le\Real z\le\frac12+t\bigg\}.$$ 
In particular, $S_0=\{z\in \CC\colon \Real(z)=1/2\}$. For $1<p<\infty$ the symbol $\delta(p)$ stands for
$$\delta(p)=|1/p-1/2|.$$

Let $U$ be an open subset of $\CC.$ We denote by $H^{\infty}(U)$ the space of bounded holomorphic functions in $U$ equipped with the supremum norm. Let $m$ be a bounded holomorphic function on $U$ which is continuous on its closure $\bar{U}$ together with derivatives up to the order $k$. We define the Mikhlin-H\"ormander norm at infinity of order $k$ on $\bar{U}$ by
\begin{equation}
\label{eq:MHU}
\|m\|_{MH(\bar{U},k)}=\max_{j=0,\ldots,k}\,\sup_{\la\in \bar{U}}(1+|\la|)^{j}\bigg|\frac{d^{j}}{d\la^{j}}m(\la)\bigg|.
\end{equation}
Slightly abusing this notation we also write
\begin{equation*}
\|m\|_{MH(\RR,k)}=\max_{j=0,\ldots,k}\,\sup_{\la\in \RR}(1+|\la|)^{j}\bigg|\frac{d^{j}}{d\la^{j}}m(\la)\bigg|.
\end{equation*}

For $V \subseteq \CC$ by $\Int V$ we mean its interior.  

By $\mF$ we denote the Fourier transform on $\RR$ given by
$$\mF(f)(x)=\int_{\RR} f(y)\,e^{ixy}\,dy,\qquad f\in L^1(\RR,dx).$$

For a locally compact Hausdorff group $H$ the symbol $\Cvp{H}$ stands for the Banach space of all right convolutors of $L^p(H).$
This space comes equipped with the norm
$$
\bignorm{S}{\Cvp{H}}
:= \sup_{\|f\|_{L^p(H)}=1}\|f*_{H} S\|_{L^p(H)}.
$$

By $G$ we always mean $SL(2,\RR).$ Elements of $G$ will by denote by $x,y,$ and the Haar measure on $G$ will be denoted by $dx.$ For a function $f\colon G\to \CC$ and $x\in G$ we write $$\tilde{f}(x)=f(x^{-1})\qquad\textrm{and}\qquad f^*(x)=\bar{f}(x^{-1}).$$  We shall often abbreviate $L^p:=L^p(G).$ By $\mD(G)$ we denote the space of compactly supported smooth functions on $G,$ while $\mD'(G)$ stands for the space of distributions on $G.$

Let $B_1,B_2$ be Banach spaces. If $T\colon B_1\to B_2$ is a linear operator we denote by $\|T\|_{B_1\to B_2}$ the operator norm of $T$. If $B_1=B_2=B$ we write $\sigma_B(T)$ for the spectrum of $T$ on $B.$

The symbol $a\lesssim_{\delta} b$ stands for the inequality $a\leq C_{\delta}\,b,$ with a constant $C_{\delta}$ that depends only on $\delta.$ We abbreviate $a\lesssim  b$ when $C$ is independent of significant quantities (in particular $C$ has to be independent of $n$). 

\subsection{The group $G=SL(2,\RR)$}
\quad

We set
\begin{equation*} 
N=\left\{n_{\xi}=
\begin{pmatrix} 1&\xi/2\\
0&1 \end{pmatrix}\colon \xi \in \RR \right\}, 
\qquad \Nb=\left\{\nb_{\xi}=
\begin{pmatrix} 1&0\\
\xi/2&1 \end{pmatrix}\colon \xi \in \RR \right\}\ , 
\end{equation*}
\begin{equation*}
K=\left\{u_{\theta}=\exp(\theta X)=
\begin{pmatrix} \cos(\theta/2) & \sin(\theta/2)\\
-\sin(\theta/2) & \cos(\theta/2) \end{pmatrix}\colon \theta \in [0,4\pi) \right\} \ ,
\end{equation*}
\begin{equation*}
A=\left\{a_t=\exp(tY_2)=
\begin{pmatrix} e^{t/2}&0\\
0&e^{-t/2} \end{pmatrix}\colon t \in \RR \right\},\qquad
A^+=\{a_t\in A: t>0 \}\ .
\end{equation*}

Throughout the paper $\TT$ stands for the torus $\big\{e^{i\theta/2}:\theta\in[0,4\pi]\big\}$. For a function $f\colon \TT\to \CC$ we let
$$
\fint_{\TT} f(\theta)\,d\theta=\frac1{4\pi}\int_0^{4\pi}f(\theta)\,d\theta. 
$$ 
We also write, for $f$ defined on $K$,
$$
\int_Kf(u)\,du=\fint_{\TT} f(u_\theta)\,d\theta\ .
$$ 
The convolution $*$ always means convolution on $G,$ i.e. 
$$f*g(x)=\int_{G}f(y)\,g(y^{-1}x)\,dy.$$

The group $G$ admits a Cartan decomposition $G=KA^+K.$ The corresponding integration formula reads
\begin{equation}
\label{eq:Cartan}
\int_{G}f(x)\,dx=\fint_{\TT}\fint_{\TT}\int_0^{\infty}\,f(u_{\psi}a_tu_{\theta})\sinh t\,dt\, d\psi\,d\theta.\end{equation}
Note that \eqref{eq:Cartan} leads to
\begin{equation*}
\int_{G}f(x)\,dx=\frac12\fint_{\TT}\fint_{\TT}\int_{\RR}\,f(u_{\psi}a_tu_{\theta})|\sinh(t)|\,dt\, d\psi\,d\theta.\end{equation*}

W also have Iwasawa decompositions $G=NAK=\bar{N}AK.$ In Iwasawa coordinates the integration formula becomes
\begin{equation*}
\int_{G}f(x)\,dx=\fint_{\TT}\int_{\RR}\int_{\RR}\,f(n_{\xi}a_tu_{\psi})e^t\, d\psi\,dt\,\,d\xi=\fint_{\TT}\int_{\RR}\int_{\RR}\,f(\nb_{\xi}a_tu_{\psi})e^t\, d\psi\,dt\,\,d\xi.\end{equation*} 	

We denote by $\Omega$ the Casimir operator
$$
\Omega=X^2-Y_1^2-Y_2^2\ ,
$$
with $X, Y_1,Y_2$ defined in \eqref{eq:XYdef}.

\medskip

\section{Spherical analysis of $K$-central functions}
\label{sec:Kcent}

\subsection{Spherical functions}\quad

Let $G$ be a Lie group and $K$ a compact subgroup of $G$. A function $f$ on $G$ is called {\it $K$-central} if $f(u^{-1}xu)=f(x)$ for every $x\in G$ and $u\in K$. A differential operator $D$ on $G$ is called $K$-central if it commutes with the inner automorphisms of $G$ induced by elements  $u\in K$.

We denote by $L^1(G)^K$ the convolution algebra of integrable $K$-central functions on $G$ and by $\DD(G)^K$ the algebra of left-invariant and $K$-central differential operators on $G$. We also set $\mD^K(G)=\mD(G)\cap L^1(G)^K$.

One says that $(G,K)$ is a {\it strong Gelfand pair} if $L^1(G)^K$ is commutative. The next two statements summarize the results about strong Gelfand pairs that will be relevant for us. We refer to \cite[Ch. 8]{Wo} for proofs and more details. 

\begin{pro}\label{strongGelfand}
Let $G$ be a connected Lie group and $K$ a compact subgroup of $G$. The following conditions are equivalent:
\begin{enumerate}
\item[\rm(i)] $(G,K)$ is a strong Gelfand pair;
\item[\rm(ii)] $\DD(G)^K$ is commutative;
\item[\rm(iii)] for every irreducible unitary representation $\pi$ of $G$, the restriction of $\pi$ to $K$ decomposes into irreducibles without multiplicities.
\end{enumerate}
\end{pro}

The pair $(G,K)$ with $G=SL(2,\RR)$ and $K=SO(2)$ is a strong Gelfand pair. This is easily seen by observing that $\DD(G)^K$ is generated by 
$$
L=-Y_1^2-Y_2^2\quad\text{ and }\quad X\ ,
$$ 
which commute with each other. They are, respectively, the $\frak p$ and the $\frak k$-components of the Casimir operator
$
\Omega=X^2-Y_1^2-Y_2^2.
$

By spherical function we mean a $K$-central function $\zeta$ which   takes the value 1 at the identity element and is an eigenfunction of all $K$-central differential operators. The following general equivalences are well known.

\begin{pro}\quad
Let $(G,K)$ be a strong Gelfand pair, with $G$ connected. The following are equivalent for a function $\zeta$ on $G$:
\begin{enumerate}
\item[\rm(i)] $\zeta$ is spherical;
\item[\rm(ii)] $\zeta$ is  $K$-central and the linear functional on $C_c(G)$
\begin{equation}\label{tildezeta}
\widetilde\zeta(f)=\int_Gf(x)\zeta(x^{-1})\,dx
\end{equation}
 is multiplicative;
\item[\rm(iii)] $\zeta$ satisfies the functional equation
$$
\int_K\zeta(uxu^{-1}y)\,du=\zeta(x)\zeta(y)\ .
$$
\end{enumerate}
\end{pro}

The spherical functions for $G=SL(2,\RR)$, $K=SO(2)$, are described as follows. We refer to \cite{L,Kn} for all unproven statements related to representation theory of $G=SL(2,\RR)$.

\begin{pro} 
For $s\in\CC$ and $n\in\frac12\ZZ$ define the following  functions on $G$:
\begin{equation*}
\begin{aligned}
&\alpha_s(n_\xi a_tu_\theta)=e^{st}\ ,\qquad \chi_n(n_\xi a_tu_\theta)=e^{in\theta}\ ,\\
&\zeta_{n,s}(x)=\int_K (\alpha_s\chi_n)(uxu^{-1} )\,du=\fint_\TT (\alpha_s\chi_n)(u_\theta x)e^{-in\theta}\,d\theta\ .
\end{aligned}
\end{equation*}

Then $\zeta_{n,s}$ is spherical and
\begin{equation*}
\Omega\,\zeta_{n,s}=\gamma(s)\,\zeta_{n,s}\ ,\qquad X\zeta_{n,s}=in\,\zeta_{n,s} \ ,\qquad L\,\zeta_{n,s}=\big(\gamma(s)+n^2\big)\zeta_{n,s}\ ,
\end{equation*}
with $\gamma(s)$ as in \eqref{eq:defgamma}.

In particular, $\zeta_{n,s}=\zeta_{n,1-s}$ for every $s$ and, modulo this identity, they  are all the spherical functions.
\end{pro}

The proof can  be found in \cite[Prop. 1]{Tak_1}. The last part of the statement follows from the fact that two spherical functions with the same pair of eigenvalues coincide. 

The bounded spherical functions determine, via \eqref{tildezeta}, the multiplicative functionals on $L^1(G)^K$, i.e., its Gelfand spectrum, denoted by $\Sigma$. The  characterization of bounded spherical functions is the first part of the next statement. Though this is known, a sketch of the proof of the first part is contained in the remark following the proof of Lemma \ref{lem:spher_decom_large_t}.

\begin{pro}
	\label{pro:Gelf}
The Gelfand spectrum $\Sigma$ consists of the spherical functions $\zeta_{n,s}$ with
\begin{itemize}
\item $0\le\Real s\le1$,
\item $s\in\big\{-|n|+1,-|n|+2,\dots,|n|\big\}$.
\end{itemize}

A spherical function $\zeta_{n,s}$ is of positive type if and only if one of the following conditions is satisfied:
\begin{itemize}
\item $\Real s=1/2$,
\item $s\in[0,1]\cup\big\{-|n|+1,-|n|+2,\dots,|n|\big\}$.
\end{itemize}
\end{pro}

The map $\zeta_{n,s}\longmapsto \big(n,\gamma(s)+n^2\big)$ establishes a 1-to-1 correspondence between $\Sigma$ and the set 
\begin{equation}\label{Sigma}
\Delta=\big\{\big(n,\gamma(s)+n^2\big):\zeta_{n,s}\in \Sigma\big\}\subset \RR\times\CC\ .
\end{equation}
By \cite{FeR}, this map is a homeomorphism. 

The spherical transform of a $K$-central integrable function $f$ on $G$ is defined on $\Sigma$ as
$$
\hat f(\zeta_{n,s})=\int_Gf(x)\zeta_{n,s}(x^{-1})\,dx\ .
$$
We will write $\hat f(n,s)$ instead of $\hat f(\zeta_{n,s})$, for
$$
(n,s)\in\big\{(n,s):s\in S_{1/2}\big\}\cup\big\{(n,s):s\in D_n\big\}\ .
$$
The function $\hat f(n,s)$ is continuous, holomorphic on $ \Int S_{1/2}$ and satisfies the identity
$$
\hat f(n,s)=\hat f(n,1-s)\ .
$$

\subsection{Connections with representation theory and Plancherel-Godement formula}\quad

Consistently with Proposition \ref{strongGelfand}(iii), for each irreducible unitary representation $\pi$ of $G$, the representation space $H^\pi$ admits an orthonormal basis $\{v^\pi_n\}_{n\in E^\pi}$, where $E^\pi\subseteq \frac12\ZZ$ and 
\begin{equation*}
\pi(u_\theta)v^\pi_n=e^{in\theta}v^\pi_n\ .
\end{equation*}
The diagonal matrix coefficients $\eta^\pi_{nn}$, where
\begin{equation*}
\eta^\pi_{jk}(x)=\langle\pi(x)v^\pi_k,v^\pi_j\rangle
\end{equation*}
give all the spherical functions of positive type, with eigenvalues
$$
d\pi(-iX)\eta^\pi_{nn}=n\,\eta^\pi_{nn}\ , \qquad d\pi(\Omega)\eta^\pi_{nn}=\omega(\pi)\,\eta^\pi_{nn}\ ,
$$
where $\omega (\pi)$ is the scalar such that $d\pi(\Omega)=\omega(\pi)I$.

Restricting the Plancherel and inversion formulas to $K$-central functions, we obtain the corresponding formulas for the spherical transform. We recall that the Plancherel measure is concentrated on the representations belonging to the unitary principal series or to the discrete series. 

The first class of representations are usually parametrized by an imaginary parameter $i\lambda$ with $\lambda>0$ and a signum $\pm$. We choose the parameters so that
$$
\omega(\pi_{i\lambda}^\pm)=\gamma \Big(\frac12+ i\lambda\Big)=\lambda^2+\frac14\ , \qquad E^{\pi_{i\lambda}^+}=\ZZ\ ,\quad E^{\pi_{i\lambda}^-}=\ZZ+\frac12\ .
$$

We  parametrize the second class of representations by $s\in\frac12\ZZ$, $s\ge1$ and a signum $\pm$ so that
$$
\omega(\pi_s^\pm)=\gamma(s)=s-s^2\ , \qquad E^{\pi_s^+}=s+\NN\ ,\quad E^{\pi_s^-}=-s-\NN\ .
$$

We define $\Sigma'\subset\Sigma$ as the set of spherical functions which are diagonal entries of representations in these two classes and $\Delta'\subset \Delta$ according to \eqref{Sigma}, i.e., 
$$
\Delta'=\big\{\big(\lambda^2+\frac14,n\big):n\in\frac12\ZZ\,,\,\lambda>0\big\}\cup \big\{(s-s^2,n):s\in D_n\big\}\ .
$$
Setting 
$$
\pi(f)=\int_G f(x)\pi(x^{-1})\,dx\ ,
$$
 and 
\begin{equation}
\label{eq:nupmn}
\nu^+(\lambda)=\lambda\tanh (\pi\lambda)\ ,\qquad \nu^-(\lambda)=\lambda\coth (\pi\lambda)\ ,
\end{equation}
the Plancherel formula is 

\begin{equation}\label{plancherel}
\begin{aligned}
\|f\|_2^2&=\frac1{2\pi}\int_0^{+\infty}\big\|\pi_{i\lambda}^+(f)\big\|_{HS}^2\,\nu^+(\lambda)\,d\lambda\\
&\quad+\frac1{2\pi}\int_0^{+\infty}\big\|\pi_{i\lambda}^-(f)\big\|_{HS}^2\,\nu^-(\lambda)\,d\lambda\\
&\quad +\frac1{2\pi}\sum_{s\in\frac12\NN+1}\Big(s-\frac12\Big)\Big(\big\|\pi_s^+(f)\big\|_{HS}^2+\big\|\pi_s^-(f)\big\|_{HS}^2\Big)\\
&\overset{\rm def}= \int_{\Delta'}\big\|\pi(f)\big\|_{HS}^2\,d\mu(\pi)\ .
\end{aligned}
\end{equation}

If $f$ is $K$-central, $\pi(f)$ is diagonal in the basis $\{v_n^\pi\}_{n\in E^\pi}$. Hence, regrouping the different terms appropriately, we obtain the following formulae.

\begin{pro}\label{central formulas}
\quad
\begin{enumerate}
\item[\rm(i)] {\rm($K$-central Plancherel-Godement formula)}  For $f\in (L^1\cap L^2)(G)^K$, we have the identity
\begin{equation}\label{central-plancherel}
\begin{aligned}
\|f\|_2^2&=\frac1{2\pi}\sum_{n\in\ZZ}\int_0^{+\infty}\big|\hat f(\zeta_{n,\frac12+i\lambda})\big|^2\,\nu^+(\lambda)\,d\lambda\\
&\quad+\frac1{2\pi}\sum_{n\in\frac12+\ZZ}\int_0^{+\infty}\big|\hat f(\zeta_{n,\frac12+i\lambda})\big|^2\,\nu^-(\lambda)\,d\lambda\\
&\quad +\frac1{2\pi}\sum_{n\in\frac12\ZZ}\,\sum_{s\in D_n}\Big(s-\frac12\Big)\big|\hat f(\zeta_{n,s})\big|^2\\\
&\quad\overset{\rm def}=\int_{\Delta_+} \hat f\,d\tilde\nu\ ,
\end{aligned}
\end{equation}
and the spherical transforms extends to an isometry of $L^2(G)^K$ onto $ L^2(\Delta_+,\tilde\nu)$.
\item[\rm(ii)] {\rm($K$-central inversion formula)}. For $f\in \mD^K(G)$, we have the identity
\begin{equation}\label{central-inversion}
\begin{aligned}
f(x)&=\frac1{2\pi}\sum_{n\in\ZZ}\int_0^{+\infty}\hat f(\zeta_{n,\frac12+i\lambda})\zeta_{n,\frac12+i\lambda}(x)\,\nu^+(\lambda)\,d\lambda\\
&\quad+\frac1{2\pi}\sum_{n\in\frac12+\ZZ}\int_0^{+\infty}\hat f(\zeta_{n,\frac12+i\lambda})\zeta_{n,\frac12+i\lambda}(x)\,\nu^-(\lambda)\,d\lambda\\
&\quad +\frac1{2\pi}\sum_{n\in\frac12\ZZ}\,\sum_{p\in D_n}\Big(s-\frac12\Big)\hat f(\zeta_{n,s})\zeta_{n,s}(x)\\
&=\int_{\Delta_+}\hat f(\zeta)\zeta(x)\,d\tilde\nu(\zeta)\ .
\end{aligned}
\end{equation}
\end{enumerate}
\end{pro}

%

\subsection{Restriction of the spherical transform to $K$-types}\label{subsec-Ktypes}\quad

For $n\in\frac12\ZZ$, we denote by $V_n$ the space of distributions $\Phi$ on $G$ of $K$-type $n$, i.e., such that
$$
R_{u_\theta}\Phi=e^{in\theta}\Phi\ ,
$$
and by $A_n$ the space of distributions in $V_n$ which are $K$-central, i.e., which satisfy the identity
$$
L_{u_{-\psi}}R_{u_\theta}\Phi=e^{in(\psi+\theta)}\Phi\ .
$$
Here $R$ and $L$ denote the right  and left regular representation respectively.

We also set $V_n^p=V_n\cap L^p$, $A_n^p=A_n\cap L^p$.
In particular, $V_n^2$ is the eigenspace in $L^2$ of $-iX$ relative to the eigenvalue $n$ and
$$
L^2(G)={\sum_{n\in\frac12\ZZ}}^{\!\!\oplus}\, V_n^2\ .
$$

The orthogonal projection $\mP_n$ of $L^2$ onto $V_n^2$, given by the formula
\begin{equation*}
\mP_n(f)(g u_{\varphi})=\fint_{\TT}f(gu_{\theta})e^{-in\theta}\,d\theta\,e^{in\varphi},\qquad g\in G\ ,\varphi \in \TT,
\end{equation*}
extends to general distributions, mapping $\mD'(G)$ onto $V_n$. If $\Phi$ is a $K$-central distribution, then 
$\mP_n(\Phi)\in A_n$. Moreover, $\mP_n$ is a contraction of all $L^p$ spaces.

For each $n\in \frac12 \ZZ$ the space $A^1_n$ is an ideal of $L^1(G)^K$. Its spectrum $\Sigma_n$ consists of the spherical functions $\zeta_{n,s}\in\Sigma$, i.e., those which are bounded and have eigenvalue $in$ relative to $X$. In accordance with \eqref{Sigma}, $\Sigma_n$ is homeomorphic to
\begin{equation}\label{Delta_n}
\Delta_n=\big\{\gamma(s):\zeta_{n,s}\in\Sigma_n\big\}=\big\{\gamma(s)+n^2:s\in S_{1/2}\cup D_n\big\}\ .
\end{equation}

\section{Estimates for spherical functions}\label{sec-estimates}

In this section we prove estimates for spherical functions $\zns(a_t)$ that will be needed later. Throughout the section we fix $n\in \ZZ/2.$ Our analysis will be based on considering separately $s\in D_n$ (discrete part, see Lemma \ref{lem:spher_demi-entier}) and $s\in S_{1/2}$ (continuous part). For $s\in S_{1/2}$ we will prove two estimates according to whether $t$ is small (see the continuous-local expansion in Lemma \ref{lem:spher_decom_small_t}) or $t$ is large (see the continuous-global expansion in Lemma \ref{lem:spher_decom_large_t}). We are able to obtain bounds that grow  at most polynomially in $n$ for both the discrete and continuous-local parts.

 Let $F(a,b,c;z)$ be the hypergeometric function, see e.g. \cite[15.2.1, p.\ 384]{NIST}. 
Using the computations in Takahashi \cite[eq. 2.19]{Tak_1} and the symmetry  $F(a,b,c;z)=F(b,a,c;z)$ we obtain, for $s\in S_{1/2}\cup D_n,$ 
\begin{equation}
\label{eq:hyper_form}
\zns(a_t)=(\cosh t/2)^{-2s} F(s-n,s+n,1,\tanh^2 t/2)=\zeta_{-n,s}(a_t),\qquad t\in \RR.
\end{equation} 
Note that the spherical function $\zns(g)$ considered by Takahshi coincides with $\zeta_{n,s}(g^{-1})$ in our notation, however this has of course no impact on \eqref{eq:hyper_form}. Using \cite[Remarque 2, p.\ 69]{Tak_1} the formula \eqref{eq:hyper_form} can be also written as
\begin{equation}
\label{eq:Tsche_form_0}
\zns(a_t)=\frac1{2\pi}\int_{-\pi}^{\pi}\bigg(\frac{\cosh t/2+e^{-i\theta}\sinh t/2}{\big|\cosh t/2+e^{-i\theta}\sinh t/2 \big|}\bigg)^{2n}\,\frac{d\theta}{(\cosh t+\sinh t\cos\theta)^s},
\end{equation}
or
\begin{equation*}
\zns(a_t)=\frac1{2\pi}\int_{-\pi}^{\pi}T_{2n}\bigg(\frac{\cosh t/2+\sinh t/2\cos\theta}{\sqrt{\cosh t+\sinh t\cos\theta}}\bigg)\,\frac{d\theta}{(\cosh t+\sinh t\cos\theta)^s},
\end{equation*}
where $T_{2n}(x)$ is the Tschebyshev polynomial defined by $T_{2n}(\cos x)=\cos (2nx).$ 

Note first a bound that follows directly from comparison with $\zeta_{0,\Real s}.$ 
	\begin{lem}
		\label{cor:large_t}
	Fix $n\in\frac12 \ZZ.$ Then for fixed $t\in \RR$ the function $ \zns(a_t)$ is holomorphic on $\Int S_{1/2}$ and it holds 
\begin{equation}
\label{eq:nto0}
|\zns(a_t)|\le \zeta_{0,\Real s}(a_t).
\end{equation}	
Moreover, for $\ve>0$ we have 
\begin{equation}
\label{eq:nt1}
|\zns(a_t)|\le C_{\ve}\,(1+|s|)^{-1/2}\,e^{-t(1/2-|\Real s -1/2|)},\qquad t\ge 1/2,\quad s\in S_{1/2-\ve}. 
\end{equation}
Consequently, $\zns\in L^p$ whenever $p>2$ and $s\in S_{1/2-1/p}.$
\end{lem}
\begin{proof}
The inequality \eqref{eq:nto0} follows directly from \eqref{eq:Tsche_form_0}. Then \eqref{eq:nt1} is a consequence of \eqref{eq:nto0} and known estimates for the spherical function $\zeta_{0,\Real s}.$
\end{proof}
\begin{remark}
Lemma \ref{cor:large_t} comes in handy when we require pointwise control of $\zns.$ It is not so useful when we need more information such as an estimate for the derivative in $s.$ For instance it is well known that, for $j\in \NN,$
\begin{equation}
\label{eq:globz0est}
\big|\partial_s^j \big(e^{ts}\zeta_{0,s}(a_t)\big)\big|\lesssim_{\ve} C_{j}\, (1+|s|)^{-1/2-j}, \qquad t\ge 1/2,\quad \ve<\Real s\le 1/2.\end{equation}
The estimate \eqref{eq:globz0est} is essentially all that is needed in \cite{A1} and \cite{StTo_1} to treat the global part of the kernel of a spherical multiplier on the symmetric space $V_0$.
However, deducing such an estimate for $\zns$ in place of $\zeta_{0,s}$ does not seem possible from \eqref{eq:globz0est}. Later on using Lemma \ref{lem:spher_decom_large_t} we shall be able to deduce an estimate of the form
\begin{equation}
\label{eq:globznest}\big|\partial_s^j \big(e^{ts}\zeta_{n,s}(a_t)\big)\big|\lesssim_{\ve} C_{n,j}\, (1+|s|)^{-1/2-j}, \qquad t\ge 1/2,\quad \ve<\Real s\le 1/2,\end{equation}
at the price of a large in $n$ constant $C_{n,j}.$
\end{remark}

We now focus on $\zns$ for $s\in D_n.$ To this end we define
$$C_{n,s}:=\frac{2^{2s}\Gamma(|n|+s)}{\Gamma(|n|-s+1)\Gamma(2s)}.$$
\begin{lem}
	\label{lem:spher_demi-entier}
Let $s\in D_n$ and take $g=\ut a_t \uph,$ where $t\in [0,\infty],$ and  $\theta,\varphi \in \TT.$ Then the spherical function satisfies
	\begin{equation} \label{eq:spher_demi-entier}|\zns(g)|\leq C\,\min\bigg(C_{n,s}\,e^{-s|t|},1\bigg).\end{equation}
	Thus, if $s\in D_n$ then  $\zns\in L^q(G),$ for every $q>1,$
	and
	\begin{equation}
	\label{eq:spher_demi_entier_norm_est}
	\|\zns\|_{L^q(G)}\le C_q \,(C_{n,s})^{1/s}\le C_q\, (1+|n|).
	\end{equation}
\end{lem}
\begin{proof}
	We start with \eqref{eq:spher_demi-entier}. It is enough to show it for $g=a_t.$ We assume that $n\ge 0,$ this suffices because of \eqref{eq:hyper_form}.  Since $s\in D_n,$ the function $\zns$ is of positive type, and thus $|\zns(a_t)|\leq \zns(e)=1.$  Moreover, $s-n$ is a negative integer. We shall use the formula for Jacobi polynomials
	$$P_k^{\alpha,\beta}(2x-1)=\frac{(\alpha+1)_k}{k!}F(-k,k+\alpha+\beta+1,\alpha+1;x),\qquad 0<x<1,$$
	see eg. \cite[eq. 2.3]{Koorn_2}. Here $(\alpha+1)_k$ is the Pochhammer symbol $\Gamma(\alpha+1+k)/\Gamma(\alpha+1)$. 
	Using the above with $k=n-s,$ $\beta=2s-1,$ $\alpha=0,$ and $x=(\tanh t/2)^2$ together with \eqref{eq:hyper_form} we obtain
	$$\zns(a_t)=(\cosh t/2)^{-2s}P_{n-s}^{0,2s-1}(2(\tanh t/2)^2-1).$$
	Hence, from eg. \cite[eq. 18.14.2]{NIST} it follows that
	$$|\zns(a_t)|\le 2^{2s}\frac{(2s)_{n-s}}{(n-s)!}e^{-s|t|}=\frac{2^{2s}\Gamma(n+s)}{\Gamma(n-s+1)\Gamma(2s)}e^{-s|t|}=C_{n,s}\,e^{-s|t|},$$ 
	and thus \eqref{eq:spher_demi-entier} is proved.

	To prove \eqref{eq:spher_demi_entier_norm_est} we employ \eqref{eq:spher_demi-entier} and use Cartan coordinates \eqref{eq:Cartan} obtaining
	\begin{align*}
	\|\zns\|_{L^q(G)}^q&\le \int_{C_{n,s}e^{-st}>1}\,e^{t}\,dt + \int_{C_{n,s}e^{-st}<1}\,C_{n,s}^q\, e^{-(sq-1)t}\,dt\\
	&=\int_{0<t<s^{-1}\log(C_{n,s})}e^{t}\,dt+C_{n,s}^q\int_{t\ge s^{-1}\log(C_{n,s})}e^{-(sq-1)t}\,dt\\
	&\le 2 \exp\bigg(\frac1s\log C_{n,s}\bigg)+\frac{C_{n,s}^q}{qs-1}\exp\bigg(-(q-1/s)\log C_{n,s}\bigg)\le 2 \frac{qs}{qs-1}(C_{n,s})^{1/s}.
	\end{align*} 
	This proves the first inequality in \eqref{eq:spher_demi_entier_norm_est}. To obtain the second inequality we need to find a uniform (in $1\le s\le n$) estimate for $C_{n,s}^{1/s}$. If $n\ge s\ge n-1/2$ then $$C_{n,s}\le C\, 2^{2s}\frac{\Gamma(2n)}{\Gamma(2n-1)}\le C 2^{2s} (1+|s|),$$
	and we are done. Assume now that $1\le s\le n-1/2.$     
	By Stirling's formula we have
	\begin{equation*}
	\begin{split}
	C_{n,s}&\le C\, 2^{2s}\,\bigg(\frac{ (n+s-1)}{(n-s)(2s-1)}\bigg)^{1/2}\,e^{-(n+s-1)+(2s-1)+(n-s)}\, \frac{(n+s-1)^{n+s-1}}{(n-s)^{n-s}(2s-1)^{2s-1}} \\&\le
	C2^{2s}\, \frac{(n+s-1)^{n+s-1}}{(n-s)^{n-s}(2s-1)^{2s-1}} =C2^{2s}\,\bigg(1+\frac{2s-1}{n-s}\bigg)^{n-s}\bigg(1+\frac{n-s}{2s-1}\bigg)^{2s-1},
	\end{split}
	\end{equation*} 
	which implies
	\begin{align*}
	C_{n,s}^{1/s}\le C\,\bigg(1+\frac{2s-1}{n-s}\bigg)^{n/s-1}\bigg(1+\frac{n-s}{2s-1}\bigg)^{2-1/s}\le C(1+n)\, \bigg(1+\frac{2}{n/s-1}\bigg)^{n/s-1}\le C(1+n).
	\end{align*}
	This completes the proof of \eqref{eq:spher_demi_entier_norm_est} and also the proof of the lemma.
	
\end{proof}
Lemma \ref{lem:spher_demi-entier} gives $L^p$ bounds for $f\mapsto f* \zns,$ $s\in D_n.$ 
\begin{lem}
	\label{lem:Lp_disc_part}
	For $s\in D_n$ the convolution operator $f\mapsto  f* \zns $ is bounded on all $L^p(G),$ $1<p<\infty.$ Moreover,
	$$\|f* \zns\|_{L^p(G)}\leq C_p (1+|n|)\|f\|_{L^p(G)}.$$
\end{lem}
\begin{proof}
	By the Kunze-Stein phenomenon $L^r(G)*L^p(G)\subseteq L^p(G)$ for $1\leq r<p\leq 2.$ Combining this with Lemma \ref{lem:spher_demi-entier} we obtain the desired conlusion for $1<p\leq 2.$
	
	For $p>2$ we use duality. Indeed, if $f\in L^p,$ $p>2,$ then, for $h\in L^{p'}(G),$ $1/p+1/p'=1$ it holds
	$$
	\langle f*\zns,h\rangle_{L^2(G)}=\langle f,h*{\zns^*}\rangle_{L^2(G)}=\langle f,h*\zns\rangle_{L^2(G)}.
	$$
	The last equality above is true because $\zns$ is of positive type for $s\in D_n,$ hence $\zns=\zns^*.$ Therefore using Lemma \ref{lem:spher_demi-entier} we finish the proof of the proposition.
	
\end{proof}

In the reminder of this section we consider the spherical functions that appear in the continuous part of the decompositions \eqref{central-plancherel} and \eqref{central-inversion}; namely $\zeta_{n,1/2+i\la}.$ Recall that $n\in \ZZ/2$ is fixed, however we are keen on keeping track of the dependence on $n$ whenever possible.

We shall prove local and global expansions in $t$ for the spherical function $\zeta_{n,1/2+i\lambda}(a_t).$ 
 An important ingredient in the proofs is an expression of $\zeta_{n,1/2+i\lambda}$ in terms of the  so-called Jacobi function $\phi_{\la}^{\alpha,\beta}$ considered by Koornwinder \cite{Koorn_1}, \cite{Koorn_2}. Combining \eqref{eq:hyper_form} (for $2i\la=2s-1$) with \cite[eq. 2.7]{Koorn_2} for $\alpha=0,$ $\beta=-2n,$ and $t$ replaced with $t/2$ we obtain
 \begin{equation}
 \label{eq:sfJacForm}
 \zeta_{n,1/2+i\lambda}(a_t)=(\cosh(t/2))^{-2n}\phi_{2\la}^{(0,-2n)}(t/2).
 \end{equation}
 
 
\subsection{Local expansion}
\quad 
For $t\ge 0$ let  $$\mJ_{t}(z)=\frac{J_{t}(|z|)}{|z|^{t}}2^{t-1}\Gamma(t+1/2),\qquad z\in \RR,$$
where $J_{t}$ is the Bessel function. Note that $\mJ_t$ is even and extends to an entire function. We will need a generalization of the local expansion of $\zeta_{0,1/2+i\la}$ obtained by Stanton and Tomas \cite[Theorem 2.1]{StTo_1}.
\begin{lem}
	\label{lem:spher_decom_small_t}
	For $0\le t\leq 1$ the spherical function $\zeta_{n,1/2+i\lambda}(a_t)$ decomposes as
	\begin{equation}
	\label{eq:spher_decom_small_t}
	\zeta_{n,1/2+i\lambda}(a_t)=\,\bigg(\frac{t}{\sinh t}\bigg)^{1/2}\sum_{j=0}^2\,t^{2j}\,b_j^n(t)\mJ_{j}(\la t)+E_n(\la,t),\qquad \la \ge 0,
	\end{equation}
	where $b_0^n\equiv b_0$ is a constant independent of $n,$  while $|b_j^n(t)|\leq C(1+|n|)^4,$ $j=1,2,$ and \begin{equation}
	\label{eq:spher_decom_small_t_error}
	\int_{1}^{\infty}|E_n(\la,t)|\la\,d\la\leq C(1+|n|)^6,\qquad \textrm{ uniformly in } 0\le t\leq 1.\end{equation} 
\end{lem}
\begin{proof}
	The Jacobi function has the following integral representation,
	\begin{align*}&\frac{\pi}{2\sqrt{2}}\,(\cosh(t/2))^{-2n}\,\phi_{2\la}^{(0,-2n)}(t/2)\\
	&=\int_0^{t/2}\,\cos 2\la s(\cosh(t)-\cosh(2s))^{-1/2}F\bigg(-2n,2n,1/2;\frac{\cosh(t/2)-\cosh s}{2\cosh (t/2)}\bigg)\,ds,\end{align*}
	see \cite[eq. 2.21]{Koorn_1}. Therefore, \eqref{eq:sfJacForm} implies an analogous representation of $\zeta_{n,1/2+i\lambda},$ namely,
	
	\begin{equation}
	\label{eq:znsintexp}
	\begin{aligned}
	\zeta_{n,1/2+i\lambda}(a_t)=\frac{\sqrt{2}}{\pi}\,\int_{-t/2}^{t/2}\,\cos  2\la s\,&(\cosh(t)-\cosh(2s))^{-1/2}\\
	&F\bigg(-2n,2n,1/2;\frac{\cosh(t/2)-\cosh s}{2\cosh (t/2)}\bigg)\,ds.
	\end{aligned}
	\end{equation}
	
	Observe that either $2n$ or $-2n$ is a non-positive integer, therefore the hypergeometric function above is in fact a polynomial. More precisely, denoting $z=\frac{\cosh s}{\cosh (t/2)}$ we have 
	\begin{align*}F\bigg(-2n,2n,1/2;\frac{1-z}{2}\bigg)=2n\sum_{j=0}^{2n}(-2)^j\frac{(2n+j-1)!}{(2n-j)!(2j)!}(1-z)^j:=\sum_{j=0}^{2n}\tilde{c}_{n,j}\bigg(\frac{1-z}{2}\bigg)^{j},\end{align*}
	see  \cite[15.2.1, p.\ 384]{NIST}. Note that $$\tilde{c}_{n,0}=1,\qquad \tilde{c}_{n,1}=-8n^2,\qquad  \tilde{c}_{n,2}=\frac{8 n^2(2n-1)(2n+1)}{3}.$$ 
	An important observation for the proof of Lemma \ref{lem:spher_decom_small_t} is that $F\big(-2n,2n,1/2;(1-z)/2\big)$ is a Chebyshev polynomial, i.e.
	$$P_{2n}\bigg(\frac{1-z}{2}\bigg):=F\bigg(-2n,2n,1/2;\frac{1-z}{2}\bigg)=T_{2n}(z),$$ see 
	\cite[15.9.5, p.\ 394]{NIST}. For $w=(1-z)/2$ (which belongs to $[0,1/2]$) we denote  \begin{equation}
	\label{eq:Tayrem} 
	R_{2n}(w)=P_{2n}(w)-1-\tilde{c}_{n,1}\,w-\tilde{c}_{n,2}\,w^2=\frac12\int_0^w P_{2n}'''(t)(w-t)^2\,dt,\end{equation}
	i.e. $R_{2n}(w)$ is the third remainder in the Taylor expansion for $P_{2n}(w).$  
	Then, defining $c_{n,j}=\frac{2\sqrt{2}}{\pi}\tilde{c}_{n,j},$  $j=0,1,2,$  we have,
	\begin{align*}                                                                                       2
	\zeta_{n,1/2+i\lambda}(a_t)&=\sum_{j=0}^{2}c_{n,j}\,\int_{-t/2}^{t/2}\,\cos  2\la s\,(\cosh(t)-\cosh(2s))^{-1/2}\bigg(\frac{\cosh(t/2)-\cosh s}{2\cosh (t/2)}\bigg)^{j}\,ds\\
	&+\frac{2\sqrt{2}}{\pi}\int_{-t/2}^{t/2}\,\cos  2\la s\,(\cosh(t)-\cosh(2s))^{-1/2}R_{2n}\bigg(\frac{\cosh(t/2)-\cosh s}{2\cosh (t/2)}\bigg)\,ds\\
	&:=\sum_{j=0}^{2}c_{n,j}M_{j}(\la,t)+E^{0,n}(\la,t).
	\end{align*}
	Note that the functions $M_j,$ $j=0,1,2,$ are independent of $n.$
	
	We start with treating the error term $E^{0,n}$. We will show that it satisfies \eqref{eq:spher_decom_small_t_error}. Integrating by parts in the $s$ variable $3$ times we see that
	\begin{equation}
	\label{eq:E0nbyparts}
	\begin{split}
	&\frac{\pi}{2\sqrt{2}}E^{0,n}(\la,t)\\
	&=\frac{8}{\la^{3}}\int_{-t/2}^{t/2}\,\sin 2\la s\,\frac{d^3}{ds^3}\left[(\cosh(t)-\cosh(2s))^{-1/2}R_{2n}\bigg(\frac{\cosh(t/2)-\cosh s}{2\cosh (t/2)}\bigg)\right]\,ds.
	\end{split}\end{equation}
	Now, a computation based on \eqref{eq:Tayrem} gives
	\begin{align*}
	\frac{d}{dw}R_{2n}(w)= \int_0^w P_{2n}'''(t) (w-t)\,dt,\quad \frac{d^2}{dw^2}R_{2n}(w)= \int_0^w P_{2n}'''(t) \,dt,\quad \frac{d^3}{dw^3}R_{2n}(w)= P_{2n}'''(w). 
	\end{align*}
	Recalling that $P_{2n}(w)=P_{2n}((1-z)/2)$ is the Chebyshev Polynomial $T_{2n}(z)=\cos (n \arccos z)$ we see that
	$$\max_{0\le w< 1/2}|P_{2n}'''(w)|\lesssim |n|^6. $$
	Hence, in view of $$\frac{d}{ds}\bigg(\frac{\cosh(t/2)-\cosh s}{2\cosh (t/2)}\bigg)=-\frac{\sinh s}{2\cosh (t/2)}$$ we obtain, for $j=0,1,2,$ 
	\begin{equation*}
	\bigg|\frac{d^j}{ds^j}\left[R_{2n}\bigg(\frac{\cosh(t/2)-\cosh s}{2\cosh (t/2)}\bigg)\right]\bigg|\leq C\, (1+|n|)^6 \bigg(\frac{\cosh(t/2)-\cosh s}{2\cosh (t/2)}\bigg)^{3-j}.
	\end{equation*}
	The above bound together with  \eqref{eq:E0nbyparts} and the Leibniz rule lead to the estimate
	\begin{equation*}
	|E^{0,n}(\la,t)|\lesssim (1+|n|)^{6}\,\la^{-3} \int_{-t/2}^{t/2}(t-2|s|)^{-1/2}\,ds \le (1+|n|)^{6}\,\la^{-3},\qquad \la>0,\quad 0<t<1.  
	\end{equation*}
	This proves \eqref{eq:spher_decom_small_t_error} for $E^{0,n}$ in place of $E_n.$

	Now we pass to the main terms $M_j$, $j=0,1,2.$ Using \eqref{eq:znsintexp} we see that $M_0(\la,t)=\zeta_{0,1/2+i\la/2}(a_t).$ Since $\zeta_{0,1/2+i\la/2}$ is the spherical function on the symmetric space $G/ K,$ by \cite[Theorem 2.1]{StTo_1} the function $M_0$ has the decomposition
	\begin{equation}
	\label{eq:F0}
	M_0(\la,t)=\bigg(\frac{t}{\sinh t}\bigg)^{1/2}\sum_{j=0}^2\,t^{2j}\,a_j(t)\mJ_{j}(\la t)+E^1(\la t),
	\end{equation}
	with $a_0\equiv 1,$ $|a_j(t)|\leq C,$ and the error term $E^1(\la t)$ satisfying the estimate \cite[eq. 2.7]{StTo_1} (with $M=2$ and $n=2$). Hence, it is easy to see that
	$\int_{1}^{\infty}|E^1(\la t)|\la\,d\la\leq C,$ uniformly in $|t|\leq 1.$
	
	It remains to consider $M_1$ and $M_2.$ Let $h(z)=\sum_{k=0}^{\infty}\frac{z^k}{(2k)!,}$ $z\in \CC.$ Then $h$ is an entire function such that $\cosh (z)=h(z^2).$ Thus $$\cosh(t/2)-\cosh s=h((t/2)^2)-h(s^2)=(t^2/4-s^2)h'(t^2/4)+\frac12(t^2/4-s^2)^2 h''(t^2/4)+R(t,s),$$
	where the remainder $R(t,s)$ is an even function of $|s|<1/2$ and satisfies $|\partial_2^{j} R(t,s)|\leq C (t/2-|s|)^{3-j}$ for $j=0,1,2,$ and $|s|\leq t/2\leq 1/2,$ and $\partial_2^{j} R(t,\pm t/2)=0,$ for $j=0,1,2.$  Therefore, for each fixed $j=0,1,2,$ we have
	$$(\cosh(t/2)-\cosh s)^j=A_j(t)(t^2/4-s^2)+B_j(t)(t^2/4-s^2)^2+R_j(s,t),$$
	where the functions $A_j$ and $B_j$ are bounded for $|t|\leq 1$ and $R_j$ has the same properties as $R.$ Therefore we can rewrite $M_j,$ $j=1,2,$ as
	$$
	\begin{aligned}
	M_j(\la,t)&=\int_{-t/2}^{t/2}\,e^{2i\la s}(\cosh(t)-\cosh(2s))^{-1/2}\bigg(\frac{\cosh(t/2)-\cosh s}{2\cosh (t/2)}\bigg)^{j}\,ds\\
	&=(2\cosh (t/2))^{-j}\left[A_j(t)\int_{-t/2}^{t/2}\,e^{2i\la s}(\cosh(t)-\cosh(2s))^{-1/2}(t^2/4-s^2)\,ds\right.\\
	&\qquad\qquad\qquad\qquad +B_j(t)\int_{-t/2}^{t/2}\,e^{2i\la s}(\cosh(t)-\cosh(2s))^{-1/2}(t^2/4-s^2)^2\,ds\\
	&\left.\qquad\qquad\qquad\qquad +\int_{-t/2}^{t/2}e^{2i\la s}(\cosh(t)-\cosh(2s))^{-1/2}\,R_j(s,t)\,ds\right]\\
	&:=(2\cosh (t/2))^{-j}\left[\phantom{\int}A_j(t)G_1(\la,t)+B_j(t)G_2(\la,t)\right.\\
	&\qquad\qquad\qquad\qquad \left.+\int_{-t/2}^{t/2}e^{2i\la s}\,(\cosh(t)-\cosh(2s))^{-1/2}R_j(s,t)\,ds\right].
	\end{aligned}
	$$
	Denoting $E^{2,j}(\la,t)=\int_{-t/2}^{t/2}\,e^{2i\la s}(\cosh(t)-\cosh(2s))^{-1/2}R_j(s,t)\,ds$ and using integration by parts thrice we obtain
	$|E^{2,j}(\la,t)|\leq C |\la|^{-3},$ uniformly in $|t|\leq 1.$ Consequently, for $j=1,2,$ we have $\int_1^{\infty}|E^{2,j}(\la,t)|\la\, d\la\leq C.$
	
	To finish the proof of Lemma \ref{lem:spher_decom_small_t} it remains to treat  $(2\cosh (t/2))^{-j}A_j(t)G_1(\la,t)$ and $(2\cosh (t/2))^{-j}B_j(t)G_2(\la,t),$ for $j=1,2.$ Using the approach from Schindler \cite{Schin_1} we will show that these contribute two Bessel function $\mJ_1$ and $\mJ_2$ plus another error term $E^3.$ Denote $\tau=t^2-(2s)^2$ and set
	\begin{align*}
	r(t,\tau)&=\frac{2\cosh t -2\cosh 2s}{t^2-(2s)^2},\qquad \textrm{for }\tau\neq0,\\
	r(t,\tau)&=\frac{\sinh t}{t},\qquad \textrm{for }\tau=0,
	\end{align*}
	where $t$ and $s$ are now complex variables. For each fixed $|t|\leq 1$ the function $r(t,\tau)$ is a non-zero analytic function in $|\tau|<7\pi^2/2,$ see \cite[p. 267]{Schin_1}. Thus, within this region $r$ has analytic powers. Applying \cite[eq. 2.4.1(3), 2.4.1(4)]{Schin_1} with $m=0,$ $y=t$ and $v=2s$ we see that
	$$(\cosh(t)-\cosh(2s))^{-1/2}=(t^2-(2s)^2)^{-1/2}\sum_{l=0}^{\infty}\alpha_{l}(t)(t^2-(2s)^2)^{l},$$
	where
	$$\alpha_l(t)=\oint_{|\tau|=3\pi^2}\frac{r(t,\tau)^{-1/2}}{\tau^{l+1}}\,d\tau.$$
	Putting this expansion in the integrals defining $G_1$ and $G_2$ we obtain
	\begin{align*}G_{1}(\la,t)&=\frac14\sum_{l=0}^{\infty}\alpha_l(t)\,\int_{-t/2}^{t/2}e^{2i\la s}(t^2-(2s)^2)^{l+1/2}\,ds,\\
	G_{2}(\la,t)&=\frac1{16}\sum_{l=0}^{\infty}\alpha_l(t)\,\int_{-t/2}^{t/2}e^{2i\la s}(t^2-(2s)^2)^{l+3/2}\,ds.
	\end{align*}
	Now, change of variable gives
		\begin{align*}\int_{-t/2}^{t/2}e^{2i\la s}(t^2-(2s)^{2})^{l-1/2}\,ds=C_{l}\,t^{2l}\mJ_{l}(\la t)\end{align*} so that
	\begin{align*}G_{1}(\la,t)&=C_{0}\,t^{2}\mJ_{1}(\la t)+C_{1}\,t^{4}\mJ_{2}(\la t)+E^3(\la,t),\\
	G_{2}(\la,t)&=C_{2}\,t^{4}\mJ_{2}(\la t)+\frac14\,E^3(\la,t),\\
	\end{align*}
	whith
	\begin{equation*}E^3(\la,t)=\sum_{l=2}^{\infty}\alpha_l(t)\,\int_{-t}^{t}e^{i\la s}(t^2-s^2)^{l+1/2}\,ds.
	\end{equation*}
	Now, the Bessel functions appearing in the formulae for $G_1$ and $G_2$ combined with the Bessel functions from \eqref{eq:F0} together  enter into \eqref{eq:spher_decom_small_t}.
	
	Thus we are left with estimating $E^3.$ Integrating by parts thrice in $\la$ we obtain, for $l\geq 2,$
	$$|\int_{-t}^{t}e^{i\la s}(t^2-s^2)^{l+1/2}\,ds|\leq \frac1{|\la|^3}\int_{-t}^t |t^2-s^2|^{l-5/2}\,ds \leq \frac1{|\la|^3}\, t^{l-3/2}.$$
	Since $|\alpha_l(t)|\leq C (3\pi^2)^{-l}$ we obtain, for $|t|\leq 1$ the bounds
	$$|E^3(\la,t)|\leq C \frac1{|\la|^3}\,\sum_{l=0}^{\infty}(3\pi^2)^{-l}\leq C \frac1{|\la|^3},$$
	which implies $\int_1^{\infty}|E^3(\la,t)|\,\la\,d\la \leq C,$ for $j=1,2.$ 
	
	In summary, setting \begin{align*}E_n:=E^{0,n}+c_{n,0}E^1+\sum_{j=1}^2c_{n,j}\, (2\cosh(t/2))^{-j}\big(E^{2,j}+(A_j+B_j/4)E^3\big)\end{align*} we we obtain the decomposition \eqref{lem:spher_decom_small_t} and finish the proof of Lemma \ref{lem:spher_decom_small_t}.
	
\end{proof}

\subsection{Global expansion}
\quad

Lemma \ref{cor:large_t} gives a bound on  $\zeta_{n,1/2+i\la}(a_t)$ for large $t$ and $|\Ima \la|<1/2. $ However, for later purpose we shall need the estimate \eqref{eq:globznest}. This will be a consequence of an asymptotic expansion proved in Lemma \ref{lem:spher_decom_large_t} below.

To state the expansion we need to introduce a function $c_n$ which is an analogue of the Harisch-Chandra $c$ function on the symmetric space (when $n=0$). We define $$Q_0(\la)=Q_{1/2}(\la)\equiv \frac1{\sqrt{\pi}},$$ and, for $\Ima \la <1/2,$ 
\begin{equation}
\label{eq:Qndef}
Q_n(\la)=\frac{1}{\sqrt{\pi}}\times \begin{cases}\frac{(i\la-n+1/2)(i\la-n+3/2)\cdots (i\la-3/2)(i\la-1/2)}{(i\la+n-1/2)(i\la+n-3/2)\cdots (i\la+3/2)(i\la+1/2)}& \textrm{ when } n\in \NN\setminus\{0\}
 \\{}\\\frac{(i\la-n+1/2)(i\la-n+3/2)\cdots (i\la-2)(i\la-1)}{(i\la+n-1/2)(i\la+n-3/2)\cdots (i\la+2)(i\la+1)}&\textrm{ when } n-\frac12\in \NN\setminus\{0\}. \end{cases}.\end{equation}
Then we set
\begin{equation}
\label{eq:cndef}
c_n(\la)=Q_n(\la)\times \begin{cases}\frac{\Gamma(i\la)}{\Gamma(1/2+i\la)}& \textrm{when }n\in \NN \\{}\\\frac{\Gamma(1/2+i\la)}{\Gamma(i\la+1)}&\textrm{when } n-\frac12\in \NN \end{cases}.\end{equation}
Note that $c_n$ is then a holomorphic function for  $\Ima \la <1/2.$

\begin{lem}
	\label{lem:cnprop}
	Fix $n\in \frac12 \bN$ and $0<\ve<1/2.$ Then for $\alpha\in\{0,1,2,3\}$ there is a constant $C_{\ve,\alpha}$ such that
	\begin{equation}
	\label{eq:cnbound}
	\bigg|\frac{d^{\alpha}}{d\la^{\alpha}}[c_n(\la)]\bigg|\leq C_{\ve,\alpha}(1+|n|)^6 (1+|\la|)^{-1/2-\alpha},\qquad| \Ima \la| <\frac12-\ve.
	\end{equation}
\end{lem}
\begin{proof}

	 We have $|\Real (i\la)|=|\Ima \la|<1/2-\varepsilon,$ and thus $$\inf_{| \Ima \la| <1/2-\ve}|\Real(i\la+1/2+j)|>0,$$ for $j\in \frac12 \bZ,$ $j\neq -1/2.$ Consequently, for each $n\in \frac12 \NN,$ the function $Q_n(\la)$ is holomorphic in $\Int S_{\delta(p)}$ and satisfies
	\begin{equation}\label{eq:Qnest0}\bigg|\frac{d^{\alpha}}{d\la^{\alpha}}Q_n(\la)\bigg|\leq C_{\ve,n,\alpha}(1+|\la|)^{-\alpha},\qquad | \Ima \la| <\delta(p).\end{equation}
	
	We claim that \eqref{eq:Qnest0} can be made more precise via
	\begin{equation}\label{eq:Qnest}\bigg|\frac{d^{\alpha}}{d\la^{\alpha}}Q_n(\la)\bigg|\leq C_{\ve,\alpha}(1+|n|)^{6}(1+|\la|)^{-\alpha},\qquad | \Ima \la| <\delta(p),\end{equation}
	for $\alpha=0,1,2,3.$ In view of \eqref{eq:Qnest0} without loss of generality we may take $n\ge 3.$  	
	To prove the claim note that on $|\Ima \la|<1/2$  the function
	$$q_{j}(\la):=\frac{i\la+1/2-n+j}{i\la+1/2+n-(j-1)},\qquad j=2,\ldots,\flo{n-1}, $$
	is holomorpic and bounded by $1,$ so that  $|D^{\alpha}_{\la}q_j(\la)|\le C_{\alpha},$ for  $\alpha\in\{0,1,2\}$ and $|\Ima \la|<1/2.$   Thus, decomposing
	$$Q_{n}(\la)=\frac{(i\la+1/2-n)(i\la+1/2-n+1)}{(i\la+1/2+(n-\flo{n-1}))(i\la+1/2+(n-\flo{n-1}-1))}\times \prod_{j=2}^{\flo{n-1}}q_{j}(\la),$$
	and using the Leibniz rule together with $|q_j(\la)|\le 1,$ we obtain \eqref{eq:Qnest}. 
	
	By properties of the Gamma function it can be proved that putting both $\frac{\Gamma(i\la)}{\Gamma(i\la+1/2)}$ and $\frac{\Gamma(i\la+1/2)}{\Gamma(i\la+1)}$ in place of $c_n(\la)$ the bound \eqref{eq:cnbound} holds with $(1+|n|)^6$ replaced by $1$. Combining this observation with \eqref{eq:Qnest} and Leibniz' rule we obtain \eqref{eq:cnbound} for $c_n.$

\end{proof}

The Lemma below is an analogue of Ionescu's \cite[Proposition A.2 c)]{Ion1} in our setting. We remark that here, in contrast with Lemmas \ref{lem:Lp_disc_part} and \ref{lem:spher_decom_small_t}, we were not able the preserve the polynomial dependence in $n$ in \eqref{eq:a_large_est}. In fact, looking closely at the proof of Lemma \ref{lem:spher_decom_large_t} it can be deduced that the constant $C_{\varepsilon,n,j}$ from \eqref{eq:globznestla} has a growth that is controlled by $\Gamma(C_{\varepsilon,j}\,n^2)$ for some constant $C_{\varepsilon,j}$ independent of $n.$  In order to lower this growth to a polynomial one an improvement of the estimate \eqref{eq:GamkEst} would be needed. This amounts to getting rid of the dependence of $n$ in $b_{\alpha,n},$ $\alpha=0,1,2$ (see \eqref{eq:ban}).

 Before stating the lemma we note that \eqref{eq:globznestla} from Lemma \ref{lem:spher_decom_large_t} below coincides with \eqref{eq:globznest} once we write $s=1/2+i\la$.

\begin{lem}
	\label{lem:spher_decom_large_t}
	Fix $0<\ve<1/2$ and take $|\Ima(\la)|\le 1/2-\ve.$ 
	
	Then the spherical function $\zeta_{n,1/2+i\lambda}(a_t)$ decomposes as
	\begin{equation}
	\label{eq:spher_decom_large_t}
	\begin{split}
	\zeta_{n,1/2+i\lambda}(a_t)&=(2\,{\cosh(t/2)})^{-2|n|}e^{(|n|-1/2)t}\\
	&\times\left(c_{|n|}(\la)e^{i\la t}(1+a_{|n|}(\la,t))+c_{|n|}(-\la)e^{-i\la t}(1+a_{|n|}(-\la,t))\right),\qquad t\ge 1/2.
	\end{split}
	\end{equation}
	The function $a_{|n|}(\la,t)$ satisfies, for each $0<\sigma<1$  and $j=0,1\,\ldots,$ the bound
	\begin{equation}
	\label{eq:a_large_est}
	\big|\partial_{\la}^j\,a_{|n|}(\la,t)\big|\leq C_{\ve,n,j}\,e^{-(1-\sigma)t}(1+|\Real \la|)^{-j},
	\end{equation}
	uniformly in $t\geq 1/2.$
	Moreover, we have the estimate
	\begin{equation}
	\label{eq:globznestla}\big|\partial_{\la}^j \big(e^{t(1/2+i\la)}\zeta_{n,1/2+i\la}(a_t)\big)\big|\leq C_{\ve,n,j}\, (1+|\la|)^{-1/2-j}, \qquad t\ge 1/2. \end{equation}
\end{lem}

\begin{proof}
The inequality \eqref{eq:globznestla} follows  from Lemma \ref{lem:cnprop} together with \eqref{eq:spher_decom_large_t} and \eqref{eq:a_large_est}. Thus we only focus on proving the formula \eqref{eq:spher_decom_large_t} and the estimate \eqref{eq:a_large_est}.
	
	Using \eqref{eq:sfJacForm} our problem reduces to expanding the Jacobi function $$\phi_{\la}^{0,-2n}(t)=(\cosh(t))^{2n}\zeta_{n,1/2+i\la/2}(a_{2t}),\qquad t\ge 1/2,$$
	for $ |\Ima \la|<1-\ve$. In the proof we assume $n\geq0,$ by \eqref{eq:hyper_form} this is no loss of generality.
	
	Let
	$$\mL:=\frac{d^2}{dt^2}+(\coth t +(-4n+1)\tanh t)\frac{d}{dt}.$$
	Then $\phi_{\la}(t):=\phi_{\la}^{(0,-2n)}(t)$ is the unique solution on $\bR^+$ of the differential equation
	\begin{equation}
	\label{eq:JacDifEq}
	\mL f +(\la^2+(-2n+1)^2)f=0
	\end{equation} satisfying
	$\phi_{\la}(0)=1$ and $D_t \phi_{\la}(0)=0.$ This follows from \eqref{eq:hyper_form} by using the differential equation satisfied by the hypergeometric function, see \cite[eq.\ (2.9)]{Koorn_2}.
	
We are going to write $\phi_{\la}$ as a combination of two other solutions of \eqref{eq:JacDifEq}, for which the asymptotics at infinity can be determined. Note that \eqref{eq:JacDifEq} approaches the equation
	$$\frac{d^2}{dt^2}+2(-2n+1)\frac{d}{dt} +(\la^2+(-2n+1)^2)f=0$$
	as $t\to \infty.$ A solution of this equation is $e^{(i\la-(-2n+1))t}.$ Moreover,
	$$\tanh t-1=2\sum_{k=1}^{\infty}(-1)^ke^{-2kt},\qquad \coth t= 2\sum_{k=1}^{\infty}e^{-2kt}.$$
	Thus we look for other solutions of \eqref{eq:JacDifEq} in the form \begin{equation}
	\label{eq:Phila}
	\Phi_{\la}(t)=e^{(i\la-(-2n+1))t}\sum_{k=0}^{\infty}\Gamma_k^n(\la)e^{-2kt}:=e^{(i\la-(-2n+1))t}(1+a(\la/2,2t)),\end{equation} with $\Gamma_0^n\equiv 1.$ In the case $n=0$ the formula \eqref{eq:Phila} is essentially the Harish-Chandra asymptotic expansion. To determine $\Gamma_k(\la)=\Gamma_k^n(\la),$ $k\geq 1,$ we put $\Phi_{\la}$ in \eqref{eq:JacDifEq} and equate the coefficients in front of $e^{(i\la-(-2n+1)-2k)t}.$ Then a computation leads to the recursion
	$$k(k-i\la)\Gamma_k=-n\sum_{j=0}^{k-1}((-2n+1)+2j-i\la)\Gamma_j+(-2n+\frac12)\sum_{j=1}^{[k/2]}((-2n+1)+2(k-2j)-i\la)\Gamma_{k-2j},$$
	cf.\ \cite[p. 16]{Joh1} with $\alpha=0,$ $\beta=-2n$ and $\rho=-2n+1$. The recursion can be rewritten as
	\begin{equation}
	\label{eq:RecIonForm}
	\Gamma_k(\la)=\sum_{j=0}^{k-1}a_j^k(\la)\Gamma_j(\la),
	\end{equation}
	where, for $k\ge 1,$ we have \begin{align*} a_j^k(\la)&=\frac{2n}{k}\big(1+\frac{2j+(-2n+1)-k}{k-i\la}\big)\qquad \textrm{when}\qquad  j\neq k\mod 2,\\
	a_j^k(\la)&=\frac{1}{2k}\big(1+\frac{2j+(-2n+1)-k}{k-i\la}\big)\qquad \textrm{when}\qquad  j= k\mod 2.
	\end{align*}
	
	We claim that \begin{equation}
	\label{eq:GamkEst}
	|D^{\alpha}_{\la}\Gamma_k(\la)|\leq C_{p,\ve}\, (n^2+1)\, k^{ b_{\alpha,n}}(1+|\Real \la|)^{-\alpha},\qquad |\Ima(\la)|\le 1-\ve,
	\end{equation}
	for some $C\geq 0$ (independent of $n$) and $ b_{\alpha,n}\geq 0,$ $\alpha=0,1,\ldots,N.$
	The first step in proving \eqref{eq:GamkEst} is the bound
	\begin{equation}
	\label{eq:ajkEst}
	|D^{\alpha}_{\la}a_j^k(\la)|\leq \frac{C_{p,\ve} (n^2+1)}{k} (1+|\Real \la|)^{-\alpha},\qquad |\Ima(\la)|\le 1-\ve,
	\end{equation}
	valid for $k\ge 1.$ 
	This easily follows from $|k-i\la|\geq c_{\ve}\max (k,|\Real \la|)$ and the definition of $a_j^k.$
	Then \eqref{eq:GamkEst} can be proved essentially as \cite[Proposition A.2 c)]{Ion1}, by using an induction argument based on  \eqref{eq:RecIonForm} and \eqref{eq:ajkEst}. We repeat Ionescu's argument for the sake of completeness.
	
	We will use repeatedly the bound
	\begin{equation}
	\label{eq:IonArg_sumb}
	1+\sum_{j=1}^{k-1}j^{b}\leq \frac{k^{b+1}}{b},
	\end{equation}
	valid for all integers $k\geq 2$ and all real numbers $b\geq 4.$ First we prove \eqref{eq:GamkEst} for $\alpha=0$ by induction over $k.$  By \eqref{eq:ajkEst} we have $|\Gamma_1(\la)|\leq C_{p,\ve} (n^2+1).$ Set $b_{0,n}=C_p(n^2+1)$ and assume that \eqref{eq:GamkEst} holds for $\alpha=0$ and $j\in\{1,\ldots,k-1\}.$ Then, by \eqref{eq:ajkEst}  with $\alpha=0$, \eqref{eq:GamkEst} for $j\in\{1,\ldots,k-1\},$  and \eqref{eq:IonArg_sumb} we obtain
	$$|\Gamma_k(\la)|\leq \sum_{j=0}^{k-1}\frac{C_p(n^2+1)}{k}|\Gamma_j(\la)|\leq \frac{C_p(n^2+1)}{k}\frac{C_p(n^2+1)k^{b_{0,n}+1}}{b_{0,n}}=C_p(n^2+1) k^{b_{0,n}}.$$
	By induction we proved \eqref{eq:ajkEst} for $\alpha=0$ and $k\in \bN.$
	
	Using induction we shall now prove that \eqref{eq:GamkEst} holds for arbitrary $\alpha\leq N,$  with \begin{equation}\label{eq:ban}b_{\alpha,n}:=C_p(n^2+1) 2^{\alpha}(\alpha+1).\end{equation} To this end we fix $\alpha\le N$ and assume that \eqref{eq:GamkEst} holds for $j\in\{1,\ldots,k-1\}$. Then, by \eqref{eq:ajkEst}, \eqref{eq:GamkEst} for $j\in\{1,\ldots,k-1\},$  and \eqref{eq:IonArg_sumb} we obtain
	\begin{align*}
	&|D^{\alpha}_{\la}\Gamma_k(\la)|\leq \sum_{\beta=0}^{\alpha}{{\alpha} \choose{ \beta}}\sum_{j=0}^{k-1}|D^{\alpha-\beta}_{\la}a_j^k(\la)||D^{\beta}_{\la}\Gamma_j(\la)|\\
	&\leq 2^{\alpha}\sum_{\beta=0}^{\alpha}\sum_{j=0}^{k-1}\frac{C_p(n^2+1)}{k(1+|\Real(\la)|)^{\alpha-\beta}}\frac{C_p(n^2+1)\max(j,1)^{{ b_{\beta,n}}}}{(1+|\Real(\la)|)^{\beta}}\\
	&
	\leq 2^{\alpha}(1+|\Real (\la)|)^{-\alpha}\sum_{\beta=0}^{\alpha}(C_p(n^2+1))^2\frac{k^{{ b_{\beta,n}}}}{ b_{\beta,n}}.
	\end{align*}
	Since the function $k^x/x$ is increasing for $k\geq 2$ we further estimate
	\begin{align*}
	|D^{\alpha}_{\la}\Gamma_k(\la)|&\leq  C_p(n^2+1)\,k^{b_{\alpha,n}}\, \frac{C_p(n^2+1) 2^{\alpha}(\alpha+1)}{ b_{\alpha,n}}(1+|\Real (\la)|)^{-\alpha}\\
	&\leq C_p(n^2+1)\,k^{ b_{\alpha,n}}\, (1+|\Real (\la)|)^{-\alpha}.
	\end{align*}
	The proof of \eqref{eq:GamkEst} is thus completed.

	In view of \eqref{eq:GamkEst} the series \eqref{eq:Phila} defining $\Phi_{\la}$ is indeed convergent for $|\Ima(\la)|\le 1-\ve$ and $t\geq 1/2.$ Moreover, $$a_{|n|}(\la,t):=\sum_{k=1}^{\infty}\Gamma_k^n(2\la)e^{-kt}$$ does  satisfy \eqref{eq:a_large_est}.
	
	Since $\Phi_{\la}$ and $\Phi_{-\la}$ are two independent solutions of \eqref{eq:JacDifEq} it must be that
	$$\phi_{\la}(t)=c^+(\la)\Phi_{\la}(t)+c^-(\la)\Phi_{-\la}(t),$$
	for some functions $c^{+}$ and $c^{-}$ of $\la.$ However, $\phi_{\la}=\phi_{-\la}$ and thus $c^+(\la)=c^{-}(-\la):=c(\la)$ so that
	\begin{equation}
	\label{eq:decomphi1}
	\phi_{\la}(t)=c(\la)\Phi_{\la}(t)+c(-\la)\Phi_{-\la}(t).\end{equation}
	Thus, from the definition of $\Phi_{\la}$ in \eqref{eq:Phila} we obtain
	\begin{equation} \label{eq:c_comp_def}
	\lim_{t\to\infty}e^{(i\la+(-2n+1))t}\phi_{\la}(t)=c(-\la),\qquad \Ima\la>0.
	\end{equation}
	Moreover, $c(\la)$ is a holomorphic function for such $\la.$  
	
	Now we focus on computing $c(-\la).$ We shall prove that 
	\begin{equation}
	\label{eq:ccnform}
	c(-2\la)=2^{-2n}c_n(-\la),\qquad \Ima\la>0,\end{equation} with $c_n$ defined by \eqref{eq:cndef}. From \eqref{eq:Tsche_form_0} we have
	\begin{align*}
	\phi_{\la}(t)=&\phi_{\la}^{0,-2n}(t)=(\cosh t)^{2n}\vp_{\la/2}(a_{2t})=\\
	&=\frac{(\cosh t)^{2n}}{2\pi}\int_{-\pi}^{\pi}\bigg(\frac{\cosh t+e^{-i\theta}\sinh t}{\big|\cosh t+e^{-i\theta}\sinh t \big|}\bigg)^{2n}\,\frac{d\theta}{(\cosh 2t+\sinh 2t\cos\theta)^{1/2+i\la/2}}\\
	&=\frac{(\cosh t)^{2n}}{2\pi}\int_{-\pi}^{\pi}\,e^{-in\theta}\,\frac{\big(e^{t}\cos \frac{\theta}{2} +i e^{-t}\sin  \frac{\theta}{2}\big)^{2n}}{(e^{2t} \cos^2\frac{\theta}{2}+e^{-2t} \sin^2\frac{\theta}{2})^{n+1/2+i\la/2}}\,  d\theta\\
	&=\frac{(\cosh t)^{2n}}{\pi}\int_{-\pi/2}^{\pi/2}\,e^{-2in\theta}\,\frac{\big(e^{t}\cos \theta +i e^{-t}\sin  \theta\big)^{2n}}{(e^{2t} \cos^2{\theta}+e^{-2t} \sin^2{\theta})^{n+1/2+i\la/2}}\,  d\theta,
	\end{align*}
	and, thus, the change of variable $x=\tan \theta $ gives 
	\begin{align*}
	\phi_{\la}(t)
&=\phi_{\la}^{0,-2n}(t)=\frac{(\cosh t)^{2n}}{\pi}\int_{\bR}\bigg(\frac{1+ix}{1-ix}\bigg)^{-n}\,\frac{(e^{t}+ixe^{-t})^{2n}(1+x^2)^{-1/2+i\la/2}}{(e^{2t}+x^2 e^{-2t})^{n+1/2+i\la/2}}
	\,dx\\
	&=e^{(-i\la-1)t}\,\frac{(\cosh t)^{2n}}{\pi}\int_{\bR}\bigg(\frac{1+ix}{1-ix}\bigg)^{-n}\,\frac{(1+ixe^{-2t})^{2n}(1+x^2)^{-1/2+i\la/2}}{(1+x^2 e^{-4t})^{n+1/2+i\la/2}}
	\,dx
	\end{align*}
	Consequently,
	\begin{equation*}
	\begin{split}
	&e^{(i\la+(-2n+1))t}\phi_{\la}(t)\\
	&=(\cosh t)^{2n}e^{-2n t}\frac{1}{\pi}\int_{\bR}\frac{(1+x^2e^{-4t})^{(i\la+1)/2}}{(1+x^2)^{(1-i\la)/2}}
	\bigg(\frac{1+ix}{1-ix}\bigg)^{-n}\bigg(\frac{1+ixe^{-2t}}{1-ixe^{-2t}}\bigg)^{n}\,dx.
	\end{split}
\end{equation*}
	Thus, if $n\geq 0,$ then \eqref{eq:c_comp_def} and the dominated convergence theorem produce
	\begin{align*}c(-\la)=&\frac{1}{2^{2n} \pi}\int_{\bR}\frac{1}{(1+x^2)^{(1-i\la)/2}}
	\bigg(\frac{1+ix}{1-ix}\bigg)^{-n}\,dx\\
	=&\frac{1}{2^{2n}\pi}\int_{\bR}(1+ix)^{(i\la-1)/2-n}
	(1-ix)^{(i\la-1)/2+n}\,dx:=\frac{1}{2^{2n}}\ell(n,\la),\end{align*}
	for $0<\Ima \la .$ 
	
	Integration by parts gives
	\begin{align*}
	\ell(n,\la)&=\frac{i}{
		\pi((i\la-1)/2-n+1)}\int_{\bR}(1+ix)^{(i\la-1)/2-n+1}\,\frac{d}{dx}
	(1-ix)^{(i\la-1)/2+n}\,dx\\
	&=\frac{(i\la-1)/2+n}{\pi((i\la-1)/2-n+1)}\int_{\bR}(1+ix)^{(i\la-1)/2-n+1}\,
	(1-ix)^{(i\la-1)/2+n-1}\,dx
	\\&=\frac{(i\la-1)/2+n}{((i\la-1)/2-n+1)}\ell(n-1,\la).
	\end{align*}

	Iterating the above we obtain
	\begin{align*}2^{2n}c(-\la)&=Q_n(-\la/2)\ell(0,\la)=Q_n(-\la/2)\frac1{\pi}\int_{\RR}(1+x^2)^{-1/2+i\la/2}\,dx=\\
	&=Q_n(-\la/2)\frac1{\pi}\int_{0}^{\infty}(1+x)^{-1/2+i\la/2}x^{-1/2}\,dx=Q_n(-\la/2)\frac{\Gamma(-i\la/2)}{\Gamma(1/2-i\la/2)}\end{align*}
	when $n$ is an integer and
		\begin{align*}2^{2n}c(-\la)&=Q_n(-\la/2)\ell(1/2,\la)=Q_n(-\la)\frac1{\pi}\int_{\RR}(1+x^2)^{-1+i\la/2}(1-ix)\,dx=\\
		&=Q_n(-\la/2)\frac1{\pi}\int_{0}^{\infty}(1+x)^{-1+i\la/2}x^{-1/2}\,dx=Q_n(-\la/2)\frac{\Gamma(1/2-i\la/2)}{\Gamma(1-i\la/2)}\end{align*}
	when $n+1/2$ is an integer. Recall that $Q_n$ is defined by \eqref{eq:Qndef}. Thus, recalling the definition of $c_n$ in \eqref{eq:cndef} we obtain \eqref{eq:ccnform}.
	
	 Since $c_n(\la)$, $\Phi(\la)$ and $\phi_{\la}$ are analytic for $|\Ima \la|<1-\ve$ using \eqref{eq:decomphi1} and \eqref{eq:ccnform} we arrive at
	\begin{equation*}
	\phi_{\la}(t)=2^{-2n}c_n(\la/2)\Phi_{\la}(t)+2^{-2n}c_n(-\la/2)\Phi_{-\la}(t),\qquad |\Ima \la|<1-\ve.\end{equation*}
	In view of \eqref{eq:sfJacForm} the above formula leads in turn to \eqref{eq:spher_decom_large_t}. The proof of Lemma \ref{lem:spher_decom_large_t} is thus completed.
\end{proof}
\begin{remark1}
The computations at the end of the proof of Lemma \ref{lem:spher_decom_large_t} allow us to justify the first part of Proposition \ref{pro:Gelf}, i.e. to determine the Gelfand spectrum of $L^1(G)^K$  (which is the set of bounded spherical functions).

 Namely, in view of \eqref{eq:ccnform} a combination of \eqref{eq:c_comp_def} and \eqref{eq:sfJacForm} leads to
$$
\lim_{t\to \infty} e^{(1/2+i\la)t} \zeta_{n,1/2+i\la}(a_t)=c_{|n|}(-\la) ,\qquad \Ima \la >0.$$
Hence, writing $s=1/2+i\la$ we have
\begin{equation}
\label{eq:remimpl}
\lim_{t\to \infty} e^{st} \zeta_{n,s}(a_t)=c_{|n|}(i(s-1/2)) ,\qquad \Real s <1/2.
\end{equation}
Thus, clearly $\zeta_{n,s}=\zeta_{n,1-s}$ is bounded for $\Real s\in [0,1].$
Moreover, \eqref{eq:remimpl} implies that the spherical function $\zeta_{n,s}$ is unbounded for $\Real s <0$ unless $c_{|n|}(i(s-1/2))=0.$ According to \eqref{eq:cndef} this happens precisely when $s$ satisfies $s-|n|\in \ZZ$  and $-|n|+1\le s\le 0.$ Now, using the identity $\zns=\zeta_{n,1-s}$ we conclude that $\zns$ is unbounded for $\Real(s)\not\in [0,1]$ unless $s-|n|\in \ZZ$ and $-|n|+1\le s\le |n|.$ 
\end{remark1}
\begin{remark2}
	A closer look at the proof of Lemma \ref{lem:spher_decom_large_t} reveals that, if we restrict to $t\ge C(n^2+1),$ then the constant $C_{\varepsilon,n,j}$ is in fact independent of $n.$  
\end{remark2}

\section{Multipliers of $L$ on $V_n^p$}

  It follows from \cite{NeSt} that $L$, being left-invariant and hypoelliptic, is essentially self-adjoint\footnote{In reality,  \cite{NeSt} only considers left-invariant elliptic operators. However, it has been observed in \cite{HuJeLu} that the same proof also applies with ellipticity replaced by hypellipticity.}.  Its self-adjoint extension, with domain
  $D(L)$ equal to the set of $f\in L^2$ such that the $Lf$, defined in the sense of distributions, is also in $L^2$, will also be denoted by $L$. 

We call $L_n$ the restriction of $L$ to $V^2_n$, defined on
$$
D(L_n)=D(L)\cap V^2_n\ .
$$
 

From the Plancherel formula \eqref{plancherel} and Proposition \ref{central formulas} we easily obtain a characterization of $L^2$ spectral multipliers of $L_n.$  In what follows we denote
\begin{equation}
\label{eq:nun}
\nu_n(\la)=\nu^+(\la), \textrm{ if } n\in\ZZ, \qquad\textrm{and}\qquad \nu_n(\la)=\nu^-(\la),\textrm{ if }n\in\frac12+\ZZ.\end{equation}  
Recall that $\nu^{\pm}$ is given by \eqref{eq:nupmn}, namely $\nu^{+}(\la)=\la \tanh \pi \la$  and $\nu^{-}(\la)=\la \coth \pi \la$.
\begin{pro}\label{spectrumL_n}
The $L^2$ spectrum of $L_n$ is the set
\begin{equation*}
\Delta'_n=\Big[\frac14+n^2,+\infty\Big)\cup\big\{\gamma(s)+n^2:s\in D_n\big\}\ .
\end{equation*}
For any bounded Borel function $m$ on $\Delta'_n$, the operator
 $m(L_n):V^2_n\longrightarrow V^2_n$ is given by $m(L_n)f=f*\Phi$,
where $\Phi\in A_n$ is defined by the formula
\begin{equation}\label{Phi<->m}
\begin{split}
\Phi
&=
\frac1{2\pi}\int_0^{\infty}
m(n^2+\la^2+1/4)\zeta_{n,1/2+i\lambda}\,\nu_n(\la)\,d\la\\&\qquad +\frac1{2\pi}\sum_{s\in
D_n}\big(s-\frac12\big)m(n^2+\gamma(s))\zeta_{n,s}\ .
\end{split}
\end{equation}

Conversely, for any distribution $\Phi\in A_n$ defining a bounded convolution operator on $V^2_n$, there is a bounded Borel function $m$ on $\Delta'_n$, unique up to sets of $\nu_n$-measure zero, such that 
$$
f*\Phi=m(L_n)f\ ,\qquad \forall\,f\in V_n.
$$
\end{pro}

Notice that the half-line $\big[\frac14+n^2,+\infty)$ may include some of the points $\gamma(s)+n^2$ with $s\in D_n$, e.g., $\Delta'_n= \big[\frac14+n^2,+\infty)$ if $n=0,\pm\frac12$. Nonetheless, the sum in \eqref{Phi<->m} is extended to all of them.

Also notice that convergence of the integral in \eqref{Phi<->m} is meant in the sense of distributions. However,  mild decay assumptions on $m$ as $\lambda\to+\infty$, e.g, $m(\lambda)=O(\lambda^{-2-\epsilon})$, imply pointwise convergence.

The formula \eqref{Phi<->m} gives the splitting $\Phi=\Phi^{cont}+\Phi^{disc}$ with
\begin{equation}
\label{eq:Phisplit}
\begin{split} 
\Phi^{cont}(x)&:=(2\pi)^{-1}\int_{0}^{\infty}m(n^2+\la^2+1/4)\zeta_{n,1/2+i\lambda}(x)\,\nu_n(\la)\,d\la,\\
\Phi^{disc}(x)&:=(2\pi)^{-1}\sum_{s\in
	D_n}\big(s-\frac12\big)m(n^2+\gamma(s))\zeta_{n,s}(x),
\end{split}
\end{equation}
where $x\in G.$ Note that both $\Phi^{cont}$ and $\Phi^{disc}$ belong to $A_n$.
\medskip

In the reminder of this section we focus on the boundedness of $m(L_n)$ on $V_n^p,$ for  $p\neq 2,$ $1<p<\infty$. For notational reasons, we prefer to extend the operator to all of $L^p$ considering the composition, for a bounded function $m$ on $\Delta'_n$,
$$
T_m:=m(L_n)\mP_n= m(L_n\mP_n)\ .
$$
Though the dependence of $T_m$ on $n$ is not explicit, we will work with a fixed $n\in \frac12 \bZ$  till the end of this section.  Notice that $T_mf=f*\Phi$, with $\Phi$ given by \eqref{Phi<->m}, and that  $T_m$ is bounded on $L^2$ if and only if $m$ is bounded on $\Delta'_n$.

Throughout the paper we set
\begin{equation}\label{eq:mdef}
m_n(s):=m(n^2+\gamma(s)),\qquad s\in S_{0}\cup D_n.
\end{equation}

First note the following proposition which reduces to \cite[Theorem 1]{CS} of Clerc and Stein when $n=0.$
\begin{pro}
	\label{pro:CleSte} Assume that $T_m$ is bounded on $L^p,$ for some $p\neq 2,$ $1<p<\infty.$ Then $m_n,$ initially defined on $S_0\cup D_n,$ extends to a bounded holomorphic function in $\Int S_{\delta(p)}.$ Moreover, we have
	\begin{equation}
	\label{eq:CleSte}
	\|m_n\|_{H^{\infty}(\Int(S_{\delta(p)}))}\le \|T_m\|_{L^p\to L^p}.
	\end{equation}
\end{pro} 
\begin{proof} 
		Assume first that $p>2.$ Then the boundedness of $T_m$ on $L^p$ implies the boundedness of $T_m^*=T_{\overline m}$ on $V_n^{p'};$ moreover, $\|T_m\|_{L^p\to L^p}=\|T_m^*\|_{L^{p'}\to L^{p'}}.$ By Corollary \ref{cor:large_t} we have $\zns\in L^{p},$ for $s\in S_{\delta(p)}.$ Thus, for $f\in L^{p'}$ we have 	$\zns\cdot T_m^* f  \in L^1$ and
	$	\langle  \zns, T_m^* f\rangle_{L^2}=\langle T_m \zns,f\rangle_{L^2}.$
	Since $T_m\, \zns=m_n(s)\, \zns$ we arrive at
	$$		\langle  \zns, T_m^* f\rangle_{L^2}=m_n(s)	\langle \zns,f\rangle_{L^2},\qquad f\in L^{p'}.$$
	Now, using Morera's theorem and Corollary \ref{cor:large_t} it is easy to see that both $	\langle  \zns, T_m^* f\rangle_{L^2}$ and $\langle \zns,f\rangle_{L^2}$ are holomorphic functions on $\Int S_{\delta(p)}.$
	Since $f\in L^{p'}$ was arbitrary we conclude that $m_n(s)$ extends to a holomorphic function and that \eqref{eq:CleSte} holds. 
	
	Assume now that $1<p<2$ and that $T_m$ is bounded on $L^p.$ Then $T_m^*=T_{\overline m}$ is bounded on $L^{p'}.$ Therefore, in view of the previous paragraph the function $\overline{m_n}$ has a holomorphic extension to $\Int S_{\delta(p)}.$ Consequently, $s\mapsto m_n(\bar{s})$ is also a holomorphic function on $\Int S_{\delta(p)}.$ Moreover, $m_n(\bar{s})=m_n(s),$ when $s\in S_0 \cup D_n,$ hence, $m_n(\bar{s})$ extends $m_n,$ and the proof is completed. 
\end{proof} 

Taking into account Proposition \ref{pro:CleSte} in what follows we assume that $m_n$ extends to a bounded holomorphic function in $\Int S_{\delta(p)}.$ We will also need to impose that $m$ has continuous derivatives on the boundary of $S_{\delta(p)}$ up to the order $3.$
Recall that in \eqref{eq:MHU} we have defined
\begin{equation}
\label{eq:MH}
\|m_n\|_{MH(S_{\delta(p)},2)}=\max_{j=0,1,2}\,\sup_{\la\in S_{\delta(p)}}(1+|\la|)^{j}\bigg|\frac{d^{j}}{d\la^{j}}m_n(\la)\bigg|,
\end{equation}
i.e.\ $\|m_n\|_{MH(S_{\delta(p)},2)}$ is the Mikhlin-H\"ormander norm of $m_n$ on $S_{\delta(p)}$ of order $3.$

The main theorem we prove is a Mikhlin-H\"ormander type multiplier theorem for $T_m.$ In the case when $n=0$ it coincides with multiplier results on the symmetric space $G/ K$ obtained by Anker \cite{A1} and Stanton and Tomas \cite{StTo_1}. 
\begin{thm}
	\label{thm:multVnp}
Fix $1<p<\infty,$ $p\neq 2.$ Assume that the function $m_n$ given by \eqref{eq:mdef} satisfies $\|m_n\|_{MH(S_{\delta(p)},2)}<\infty.$ Then $T_m$ is a bounded operator on $L^p.$ Moreover,
\begin{equation}
\label{eq:multVnp}
\|T_m f\|_{L^p}\leq C_{p,n}\,\bigg(\|m_n\|_{MH(S_{\delta(p)},2)}+ \sum_{s\in D_n}s\,|m_n(s)|\bigg)\|f\|_{L^p}
\end{equation}
\end{thm}
\begin{remark1}
In \cite[Theorem 8]{I2} Ionescu proved that for $n=0$ one can replace \eqref{eq:MH} with $$\max_{j=0,1,2}\,\sup_{\la\in S_{\delta(p)}}(|\la\pm i\delta(p)|)^{j}\bigg|\frac{d^{j}}{d\la^{j}}m_0(\la)\bigg|<\infty,$$
which is clearly weaker than $\|m_n\|_{MH(S_{\delta(p)},2)}<\infty$. 
An important ingredient in the proof of \cite[Theorem 8]{I2} was a connection between Cartan and Iwasawa coordinates. This is very useful on $V_0,$ but not so much on $V_n$ for $n\neq 0.$ The problem is that the convolution kernels of the operators $T_m$ are no more bi-$K$ invariant. Therefore the argument from \cite{I2} does not seem to work in our setting. 
\end{remark1}
\begin{remark2}
By the change of variable $z=\gamma(s)+n^2$ one can restate the condition \eqref{eq:MH} into a Mikhlin-H\"ormander condition for $m(z)$ on the image $\gamma[S_{\delta(p)}]+n^2$ (which is a shifted parabola in the right half plane). However, on symmetric spaces (of which the space $V_0$ is a particular instance) it is customary to state the Mikhlin-H\"ormander condition on a strip rather than a parabola. Therefore we decided to follow this convention also for general $n\in\ZZ/2.$
\end{remark2}

Before proceeding to the proof of Theorem \ref{thm:multVnp} we remark that an approximation argument shows that we may assume that $m_n(s)$ has rapid decay when $|s|\to \infty.$ Indeed, let us replace $m$ with  $$m^{\varepsilon}(z)=m(z)e^{-\varepsilon z},\qquad z\in \Delta'_n.$$
Then, since $T_{m^{\varepsilon}}=T_m \,e^{-\varepsilon L},$ and the heat semigroup $\{e^{-tL}\}_{t>0}$ is strongly continuous on $L^p,$  we have $$\lim_{\ve\to0^+}T_{m^{\varepsilon}}=T_m,\textrm{ strongly in $L^p$.}$$
Since also
$$\lim_{\varepsilon\to 0^+}\|(m^{\varepsilon})_n\|_{MH(S_{\delta(p)},2)}=\|m_n\|_{MH(S_{\delta(p)},2)},$$ coming back to \eqref{eq:multVnp} we may indeed assume that $m_n$ has the desired decay.

In view of the previous paragraph from now on we assume that $m_n(s)$ vanishes exponentially when $|s|\to \infty.$

The first step in the proof of Theorem \ref{thm:multVnp} is the splitting $T_m f= T^{cont}f+T^{disc }f$ where
$$T^{cont}f=f*\Phi^{cont},\qquad T^{disc }f=f*\Phi^{disc}$$
with the distributions $\Phi^{cont}$ and $\Phi^{disc}$ given by \eqref{eq:Phisplit}, i.e.
$$\Phi^{cont}(x)=(2\pi)^{-1}\int_{0}^{\infty}m_n(1/2+i\la)\,\zeta_{n,1/2+i\lambda}(x)\,\nu_n(\la)\,d\la$$
and
$$\Phi^{disc}(x)=(2\pi)^{-1}\sum_{s\in
	D_n}\big(s-\frac12\big)m_n(s)\zeta_{n,s}(x).$$
We remark that by the approximation assumptions on $m_n$ both $\Phi^{cont}$ and $\Phi^{disc}$ are actually bounded functions on $G.$

Lemma \ref{lem:Lp_disc_part} gives the following bound for the discrete part.
\begin{pro}
	\label{pro:disc}
	The operator $T^{disc}f=f*\Phi^{disc}$ is bounded on $L^p;$ moreover, 
	\begin{equation*}
	\|T^{disc} f\|_{L^p}\leq C_p (1+|n|)\,\sum_{s\in D_n}s\,|m_n(s)|\,\|f\|_{L^p},\qquad 1<p<\infty. 
	\end{equation*}
\end{pro}

In view of Proposition \ref{pro:disc} the proof of Theorem \ref{thm:multVnp} will be completed once we have the following.
\begin{pro}
	\label{pro:Cont}
Fix $1<p<\infty$ and assume that the function $m_n$ given by \eqref{eq:mdef} extends to a bounded holomorphic function in $\Int S_{\delta(p)}$ which satisfies $\|m_n\|_{MH(S_{\delta(p)},2)}<\infty.$ Then $T^{cont}$ is a bounded operator on $L^p;$ moreover
\begin{equation*}
\|T^{cont} f\|_{L^p}\leq C_{p,n}\,\big(\|m_n\|_{MH(S_{\delta(p)},2)}\big)\|f\|_{L^p}.
\end{equation*} 
\end{pro}

The reminder of this section is devoted to the proof of Proposition \ref{pro:Cont}. Note that the spherical functions $\zeta_{n,1/2+i\lambda}$ are bounded for $|\Ima(\la)|< 1/2$ and are even with respect to $\la$ from the real line. Thus, using \eqref{eq:mdef}, for $x\in G$  we have
\begin{equation}
\label{eq:Phiform1}
\Phi^{cont}(x)=\frac12\,\int_{\RR}m_n(1/2+i\la)\,\zeta_{n,1/2+i\lambda}(x)\, \nu_n(\la)\,d\la.\end{equation}

The proof of Proposition \ref{pro:Cont} is based on splitting the kernel $\Phi^{cont}$ into local and global parts in the variable $a_t$ in Cartan coordinates. More precisely, let $\chi$ be a smooth even function on $\RR$ such that $\chi(t)=1$ for $|t|<1/2,$ and $\chi(t)=0,$ for $|t|>1$. For $x=u_{\psi}a_t u_{\theta}$ with $t\in \RR$ we split
$$\Phi^{cont}(x)=\chi(t)\Phi^{cont}(x)+(1-\chi(t))\Phi^{cont}(x):=\Phi^{loc}(x)+\Phi^{glo}(x).$$ Then $\Phi^{loc} $ and $\Phi^{glo}$ are still in $A_n^{\infty}.$ 

We remark that a splitting into local and global parts at the level $n^2+1$ instead of the at the level $1$ would give the explicit estimate $e^{Cn^2}$ for the constant $C_{p,n}$ from Proposition \ref{pro:Cont}. Unfortunately, this ruins the polynomial estimate for the continuous-local part.  

In the reminder of this section we treat separately the operators $T^{loc}f=f*\Phi^{loc}$ and $T^{glo}f=f*\Phi^{glo}.$

\subsection{The continuous local part}
\quad

In this section we demonstrate the bound
\begin{equation}
\label{eq:Tlocest}
\|T^{loc} f\|_{L^p}\leq C_p \, (1+|n|)^{8}\big(\|m_n\|_{MH(S_{\delta(p)},2)}\big)\|f\|_{L^p};
\end{equation}
note that the dependence on $n$ in \eqref{eq:Tlocest} is polynomial.
Here we need a Coifman-Weiss type transference result.  The proof of Lemma \ref{lem:CWtrans} is similar to the one given in   \cite[Theorem 8.7]{CWtr} by Coifman and Weiss. However, we give it for the sake of completeness.
\begin{lem}
 \label{lem:CWtrans}
Fix $1\leq p<\infty$. Let $\Psi$ be a compactly supported function which belongs to $A_n^{\infty}$. If
convolution with $S(t):=|\sinh(t)|\, \Psi(a_t)$ is a bounded operator
on $L^{p}(\RR)$ then convolution with $\Psi$ is a bounded
operator on $L^p(G).$ Moreover, we have
\begin{equation*}
\|f*\Psi\|_{L^p}\leq  \|S\|_{\Cvp{\RR}}\,\|f\|_{L^p}.
\end{equation*}
\end{lem}
\begin{proof}
Since $f*\Psi=\mP_n f *\Psi$ for $\Psi\in A_n$ we may assume that $f\in V_n^p.$ Then, by the assumption that $\Psi \in A_n$ (second equality below), Minkowski's integral inequality (first inequality below), and right invariance of $dx$   we obtain
\begin{align*}
&\left\|f*\Psi\right\|_{L^p}=\bigg(\int_{G}\bigg|\int_{G} f(xy^{-1})\Psi(y)\,dy\bigg|^p\,dx\bigg)^{1/p}\\
&=\bigg(\int_{G}\bigg|\fint_{\TT}\fint_{\TT}\int_{\RR} f(xu_{\-\varphi}a_{-t}u_{-\theta})\Psi(u_{\theta}a_{t}u_{\varphi})\,|\sinh(t)|\,dt\,d\varphi\,d\theta\bigg|^p\,dx\bigg)^{1/p}\\
&=\bigg(\int_{G}\bigg|\fint_{\TT}\int_{\RR} f(xu_{\-\varphi}a_{-t})\Psi(a_{t}u_{\varphi})\,|\sinh(t)|\,dt\,d\varphi\bigg|^p\,dx\bigg)^{1/p}\\
&\le \fint_{\TT} \bigg(\int_{G}\bigg|\int_{\RR} f(xu_{\-\varphi}a_{-t})\Psi(a_{t}u_{\varphi})\,|\sinh(t)|\,dt\,\bigg|^p\,dx\bigg)^{1/p}\,d\varphi\\
&= \bigg(\int_{G}\bigg|\int_{\RR} f(xa_{-t})\Psi(a_{t})\,|\sinh(t)|\,dt\,\bigg|^p\,dx\bigg)^{1/p}.
\end{align*}
Hence, letting $\mR$ to be the representation of $\RR$ acting on $G$ by   $\mR_{-t}f(x)=f(xa_{-t})$ we have proved that
\begin{align*}
&\left\|f*\Psi\right\|_{L^p}\le   \bigg(\int_{G}\bigg|\int_{\RR} \mR_{-t}f(x)\Psi(a_{t})\,|\sinh(t)|\,dt\,\bigg|^p\,dx\bigg)^{1/p}.
\end{align*}

At this point an application of the Coifman-Weiss transference principle \cite[Theorem 2.4]{CWtr} completes the proof of  Lemma \ref{lem:CWtrans}.
\end{proof}

Let $S(t):=|\sinh(t)|\Phi^{loc}(a_t),$ $t\in \RR.$ Lemma \ref{lem:CWtrans} reduces \eqref{eq:Tlocest} to the following estimate.
\begin{lem}
 \label{lem:part_Loc}
If $\|m_n\|_{MH(S_{\delta(p)},2)}<\infty$ then $f\mapsto f*_{\RR}S$ is a bounded operator on all $L^q(\bR),$ $1<q<\infty.$ Moreover, we have
\begin{equation}
\label{eq:Tloctransf}
\|S\|_{\Cvq{\RR}}\le C_q (1+|n|)^{6}\, \|m_n\|_{MH(S_{\delta(p)},2)}. 
\end{equation}
\end{lem}
\begin{proof}
We abbreviate $\tm(\la)=m_n(1/2+i\la).$ Note that $\tm$ is then an even function of $\la.$ Since 
\begin{equation}
\label{eq:mcomp}
\|\tm\|_{MH(\RR,2)}\le C_p \|m_n\|_{MH(S_{\delta(p)},2)}\end{equation}
the lemma will be proved if we obtain \eqref{eq:Tloctransf} with $\|\tm\|_{MH(\RR,2)}$ in place of $\|m_n\|_{MH(S_{\delta(p)},2)}.$  	
	
The proof of Lemma \ref{lem:part_Loc} is based on the local expansion of $\zeta_{n,1/2+i\lambda}$ from Lemma \ref{lem:spher_decom_small_t}. 

Let $\eta$ be an even $C_c^{\infty}$ function equal to $1$ on $[-1,1]$ and equal to $0$ on $\RR\setminus [-2,2].$ Then
\begin{equation*}S(t)=S_1(t)+S_2(t)\end{equation*} 
with
\begin{align*}
S_1(t)&:=\frac{ \chi(t)|\sinh t|}{2} \int_{\RR} \eta(\la)\tm(\la)\zeta_{n,1/2+i\lambda}(a_t)\nu_n(\la)\,d\la\\
S_2(t)&:= \frac{  \chi(t)|\sinh t|}{2}  \int_{\RR} (1-\eta)(\la)\tm(\la)\zeta_{n,1/2+i\lambda}(a_t)\nu_n(\la)\,d\la
\end{align*}

In view of $|\zeta_{n,1/2+i\lambda}(a_t)|\le 1$ we have $|S_1(t)|\leq C\, \|\tm\|_{L^\infty(\RR)}.$ Since $S_1$ is compactly supported we thus obtain
   $\|S_1\|_{L^1(\RR)}\leq C\, \|\tm\|_{MH(\RR,2)}$ and, consequently,
   \begin{equation}
   \label{eq:locS1est}
   \|S_1\|_{\Cvq{\RR}}\leq C  \|\tm\|_{MH(\RR,2)}.
   \end{equation}

It remains to consider convolution with $S_2.$ Observe that $\zeta_{n,1/2+i\la}(a_t)$ is an even function of both $t$ and $\la.$ 
Therefore applying Lemma \ref{lem:spher_decom_small_t} we may split $S_2=\sum_{j=0}^2S_2^j+S_2^E$ where
\begin{equation*}
S_2^j(t):=  |\sinh t|^{1/2}|t|^{1/2+2j}b_j^n(|t|)\chi(t)  \int_{0}^{\infty} (1-\eta)(\la)\tm(\la)\,\mJ_{j}(\la t )\nu_n(\la)\,d\la,\qquad t\in \RR,
\end{equation*}
for $ j=0,1,2,$ and
\begin{equation*}
S_2^E(t):= |\sinh t|^{1/2}|t|^{1/2}\chi(t)  \int_{0}^{\infty} (1-\eta)(\la)\tm(\la)\,E_n(\la,|t|)\,\nu_n(\la)\,d\la,\qquad t\in \RR,
\end{equation*}
Note that $(1-\eta)\tm$ also satisfies $$\|(1-\eta)\tm\|_{MH(\RR,2)}<C\|\tm\|_{MH(\RR,2)}<\infty;$$ moreover, it vanishes for $|\la|<1.$
Now, by \eqref{eq:spher_decom_small_t_error} we have
 \begin{equation*}|S_2^E(t)|\leq C\|(1-\eta)\tm\|_{L^{\infty}(\RR)}\int_{1}^{\infty}|E_n(\la,|t|)|\,\la\, d\la\leq C\,(1+|n|)^{6}\|\tm\|_{MH(\RR,2)},\qquad |t|\leq 1.\end{equation*}  Since $S_2^E$ is compactly supported, we thus obtain
 \begin{equation}
 \label{eq:locS2Eest}
 \|S_2^E\|_{\Cvq{\RR}}\leq C\,(1+|n|)^{6}\|\tm\|_{MH(\RR,2)}.
 \end{equation}
Thus we are left with considering $S_2^j,$ $j=0,1,2.$

We start with $j=1,2.$ We apply the formula
$$\bigg(\frac1z\frac{d}{dz}\bigg)^{j} \mJ_{0}(z)=c_j\,\mJ_{j}(z),\qquad z\ge 0,$$
see \cite[eq.\ 10.6.6, p.\ 222]{NIST}. Integrating by parts once in $\la$ we get
\begin{equation}
\label{eq:locS2jform}
S_2^j(t)=  |\sinh t|^{1/2} |t|^{1/2} t^{2(j-1)}b_j^n(t)\chi(t)  \int_{0}^{\infty}\bigg(\frac{d}{d\la}\circ \frac1{\la}\bigg)\bigg( (1-\eta)\cdot\tm\cdot\nu_n\bigg)(\la)\,\mJ_{j-1}(\la t )\,d\la,
\end{equation}
where $t\in \RR.$
Note that $$\bigg|\bigg(\frac{d}{d\la}\circ \frac1{\la}\bigg)\bigg( (1-\eta)\cdot\tm\cdot\nu_n\bigg)(\la)\bigg|\le C (1+|\la|)^{-1}\,\|\tm\|_{MH(\RR,2)};$$
moreover we have $|\mJ_{t}(z)|\le \min (1,z^{-1/2-t}),$ for $z,t\ge 0.$ 
Therefore splitting the integral \eqref{eq:locS2jform} according to $|\la t|<1$ or $|\la t|\ge 1$ we obtain
\begin{align*}
&|S_2^j(t)|\\
&\le  \|\tm\|_{MH(\RR,2)}\,|t|^{2j-1} \int_{|\la|<1/t}(1+|\la|)^{-1}\,d\la+ \|\tm\|_{MH(\RR,2)}\,|t|^{2j-1}\int_{|\la|>1/t}|\la t|^{-1/2-j}\,d\la\\
&
\le  \bigg( |t|\log |t|+|t|^{2(j-1)}\bigg)\, \|\tm\|_{MH(\RR,2)}\le C\,\|\tm\|_{MH(\RR,2)},\qquad |t|\le 1,
\end{align*}
and, consequently,
\begin{equation}
\label{eq:locS2jest}
\|S_2^j\|_{\Cvq{\RR}}\le  C_q \|\tm\|_{MH(\RR,2)}.
\end{equation}
for $j=1,2.$

It remains to treat 
$$S_2^0(t)=  |\sinh t|^{1/2}|t|^{1/2}\chi(t) \int_{0}^{\infty} (1-\eta)(\la)\tm(\la)\,\mJ_{0}(\la t )\nu_n(\la)\,d\la,\qquad t\in \RR.$$
The formula \cite[eq. 10.9.12, p.\ 224]{NIST} implies
$$\mJ_0(\la t )=c\, \int_{1}^{\infty}(\xi^2-1)^{-1/2}\, \sin{\xi |t|\la}\,d\xi=c\,\int_{\la}^{\infty}(\xi^2-\la^2)^{-1/2}\, \sin{\xi |t|}\,d\xi,\qquad \la \ge 0.$$
Therefore, setting
$$g(\xi)=\int_{0}^{\xi}(\xi^2-\la^2)^{-1/2}\,(1-\eta)(\la)\tm(\la)\nu_n(\la)\, d\la,\qquad \xi\ge 0$$ 
and
$$h(\xi)=\bigg(\frac{d}{d\xi}g\bigg)(|\xi|),\qquad \xi\in \RR,$$
and using Fubini's theorem followed by integration by parts we obtain
\begin{align*}S_2^0(t)&= c|\sinh t|^{1/2}|t|^{1/2}\chi(t)  \int_{0}^{\infty} g(\xi)\,\sin \xi|t|\,d\xi=\frac{ -c|\sinh t|^{1/2}\chi(t)}{|t|^{1/2}}  \int_{0}^{\infty} \frac{d}{d\xi}g(\xi)\,\cos \xi|t|\,d\xi\\
&=\frac{ -c\,|\sinh t|^{1/2}\chi(t)}{2|t|^{1/2}}  \int_{\RR} h(\xi)
\,e^{i\xi t}\,d\xi,\qquad t\in \RR.\end{align*}
Consequently, denoting $\tilde{\chi}(t)=\frac{ -c\,|\sinh t|^{1/2}\chi(t)}{2|t|^{1/2}},$ $t\in \RR,$ we have
$\mF (S_2^0)(\xi)=(\mF(\tilde{\chi})*_{\RR}h)(\xi).$

We claim that
\begin{equation}
\label{eq:S20comp}
\sup_{\xi\in\RR}\bigg(|h(\xi)|+\bigg|\xi\frac{d}{d\xi}h(\xi)\bigg|\bigg)\leq C \|\tm\|_{MH(\RR,2)}.\end{equation}
To obtain \eqref{eq:S20comp} we change variables getting $$g(\xi)=\frac12 \int_{-1}^{1}(1-\la^2)^{-1/2}((1-\eta)\tm\nu_n)(\la \xi)\,d\la,\qquad \xi>0,$$
so that
$$h(\xi)=\frac12 \int_{-1}^{1}(1-\la^2)^{-1/2}\la \big((1-\eta)\tm\nu_n\big)'(\la |\xi|)\,d\la.$$
Now a computation produces \eqref{eq:S20comp}.

Since $\mF(\tilde{\chi})$ is a Schwarz function \eqref{eq:S20comp} remains true with $\mF(S_2^0)$ replacing $h.$ Thus, applying the Mikhlin multiplier theorem on $\RR$ we arrive at
\begin{equation}
\label{eq:locS20est}
\|S_2^0\|_{\Cvq{\RR}}\leq C_q\, \|\tm\|_{MH(\RR,2)}.
\end{equation}

In summary, combining \eqref{eq:locS1est}, \eqref{eq:locS2Eest}, \eqref{eq:locS2jest}, and \eqref{eq:locS20est}, and then, using \eqref{eq:mcomp} we obtain \eqref{eq:Tloctransf}. The proof of the lemma is thus finished.

\end{proof}
\subsection{The continuous global part}
\quad

This section is devoted to the proof of the estimate
\begin{equation}
\label{eq:Tgloest}\| f*|\Phi^{glo}|\|_{L^p}\leq C_{p,n}\big(\|m_n\|_{MH(S_{\delta(p)},2)}\big)\|f\|_{L^p}.
\end{equation}
We remark that contrary to Proposition \ref{pro:disc} and the estimate \eqref{eq:Tlocest} here we are not able to keep the explicit polynomial dependence on $n.$ This is due to a lack of such an estimate in \eqref{eq:globznestla} from Lemma \ref{lem:spher_decom_large_t}.

Duality arguments show that it is enough to take $1<p<2.$ Indeed, if $h\in L^{p'}$ then since $|\Phi^{glo}(x^{-1})|=|\Phi^{glo}(x)|,$ $x\in G,$ we have
	$
	\langle f*|\Phi^{glo}|,h\rangle_{L^2}=\langle f,h*|\Phi^{glo}|\rangle_{L^2}.
	$
Clearly it also holds $\|m_n\|_{MH(S_{\delta(p)},2)}=\|m_n\|_{MH(S_{\delta(p')},2)}.$ Thus in the reminder of this section we consider $1<p<2.$

In what follows we set $$\Phi_p(u_{\varphi}a_t u_{\theta})=e^{t/p}\Phi^{glo}(u_{\varphi}a_t u_{\theta}),\qquad t>0.$$
From the global expansion of $\zeta_{n,1/2+i\la}$ proved in Lemma \ref{lem:spher_decom_large_t} we deduce estimates that are crucial for the proof of  \eqref{eq:Tgloest}.
\begin{lem}
\label{lem:keyglob}
If $\|m_n\|_{MH(S_{\delta(p)},2)}<\infty,$ then we have
\begin{equation*}
|\Phi_p(a_t)|\leq C_{p,n}\,(1+t)^{-2} \|m_n\|_{MH(S_{\delta(p)},2)},\qquad t>0.\end{equation*}
\end{lem}
\begin{proof}
	
	 By \eqref{eq:Phiform1} we have
	$$\Phi^{glo}(a_t)=\frac12\,(1-\chi(t))\,\,\int_{\RR}\mu(\la)\,\zeta_{n,1/2+i\la}(a_t)\, d\la,\qquad t>0,$$
	where we have set $$\mu(\la)=m_n(1/2+i\la)\nu_n(\la),$$
	with $\nu_n$ defined in \eqref{eq:nun}. Note that by our assumptions on $m_n$ we have the estimate
		\begin{equation}
	\label{eq:tmprop0}
	\bigg| \frac{d^j}{d\la^j} \mu(\la)\bigg|\leq C\,(1+|\la|)^{-j+1}\,\|m_n\|_{MH(S_{\delta(p)},2)},\qquad |\Ima \la|\le \delta(p),
	\end{equation}
	for $j=0,1,2.$
	
		Changing the path of integration to $\{\la-i\delta(p)\colon \la \in \RR\}$ we obtain
		\begin{equation*}\Phi^{glo}(a_t)=\frac12\,(1-\chi(t))\,\int_{\RR}\mu(\la-i\delta(p))\,\zeta_{n,1/p+i\la}(a_t)\, d\la,\end{equation*}
	for $t>0,$ so that
		\begin{equation*}
		\Phi_p(a_t)=\frac12\,(1-\chi(t))\,\int_{\RR}\big[\mu(\la-i\delta(p))e^{t(1/p+i\la) }\zeta_{n,1/p+i\la}(a_t)\big]\,e^{-it\la} d\la.
		\end{equation*} 
	Now, \eqref{eq:tmprop0} together with \eqref{eq:globznestla} from Lemma \ref{lem:spher_decom_large_t} show that for $j=0,1,2,$ it holds
		\begin{equation}
		\label{eq:tmprop2}
		\bigg| \partial_{\la}^j\big[\mu(\la-i\delta(p))e^{t(1/p+i\la) }\zeta_{n,1/p+i\la}(a_t)\big]\bigg|\leq C_{p,n}\,(1+|\la|)^{-j+1/2}\,\|m_n\|_{MH(S_{\delta(p)},2)},
		\end{equation}
	where $\la \in \RR.$ Integrating by parts in $\la$ twice we obtain
	\begin{equation*}
	\Phi_p(a_t)=\frac{\chi(t)-1}{2t^2}\,\int_{\RR}\partial_{\la}^2\big[\mu(\la-i\delta(p))e^{t(1/p+i\la) }\zeta_{n,1/p+i\la}(a_t)\big]e^{-it \la}\, d\la.\end{equation*}
Therefore, applying \eqref{eq:tmprop2} we complete the proof of the lemma.
 \end{proof}

We are now ready to prove \eqref{eq:Tgloest}. Observe that  $|\Phi^{glo}|\in A_0.$ Hence, by the Herz majorizing principle (see \cite{Her1}) it is enough to show that
\begin{equation}
\label{eq:Herz1}
\int_0^{\infty}|\Phi^{glo}(a_t)||\sinh t|e^{-t/p'}\,dt \leq C_{p,n}\big(\|m_n\|_{MH(S_{\delta(p)},2)}\big).
\end{equation}
By definition $\Phi^{glo}(a_t)=e^{-t/p}\Phi_p(a_t),$ thus, Lemma \ref{lem:keyglob} gives
\begin{align*}
\int_0^{\infty}|\Phi^{glo}(a_t)||\sinh t|e^{-t/p'}\,dt\le \int_1^{\infty}|\Phi_p(a_t)|\,dt\, \le C_{p,n} \, \|m_n\|_{MH(S_{\delta(p)},2)}.
\end{align*}
Therefore, \eqref{eq:Herz1} is justified, and the proof of \eqref{eq:Tgloest} is finished.

\subsection{The full continuous part}
\quad

Summarizing the previous two sections, \eqref{eq:Tlocest} and \eqref{eq:Tgloest} imply
$$\|T^{cont} f\|_{V_n^p}\le \|T^{loc} f\|_{V_n^p} + \|T^{glo} f\|_{V_n^p}\le  C_{p,n}\, \|m_n\|_{MH(S_{\delta(p)},2)}\,\|f\|_{V_n^p}.$$
The proof of Proposition \ref{pro:Cont} is thus completed, hence, also the proof of Theorem \ref{thm:multVnp}.

\section{Joint spectral multipliers of $(L,-iX)$}
\label{sec:jsmLp}

We study the joint spectral multipliers of $(L,-iX)$ on $L^p,$ $1<p<\infty.$ This is done first for $p=2$ when we have a full characterization. For other values of $p$ we are able to determine the joint $L^p$ spectrum of $(L,-iX)$ and give a necessary condition for the $L^p$ boundedness of $m(L,-iX)$.


We start with the $L^2$ theory and denote by $E_{1}$ and $E_{2}$ the spectral measures of $L$ and $-iX,$ respectively. The next statement is a direct consequence of  Section \ref{subsec-Ktypes} and Proposition \ref{spectrumL_n}.

\begin{lem} 
	\label{lem:jointL2spec}
The closures of $-iX$ and $L$  strongly commute (i.e., $E_{1}$ and $E_{2}$ commute) and their joint $L^2$-spectrum is the set $\Delta(L,-iX)\subset\Delta_+$ given by
\begin{equation*}
\Delta(L,-iX)=\bigcup_{n\in\frac12\ZZ}\big(\Delta_n^2\times \{n\}\big)
\end{equation*}
with
\begin{equation*}
\Delta_n^2=\big\{z+n^2:z\in [1/4,+\infty)\cup \gamma[D_n]\big\}\ .
\end{equation*}

\end{lem}

%
%
%
%

The following statements about functional calculus can be justified in a similar way. We denote by $E$ the joint spectral measure of the pair $(L,-iX)$ which is uniquely determined by $E(\omega_1\times \omega_2)=E_1(\omega_1)E_2(\omega_2).$

\begin{pro}\label{pro-multiplier}
\quad
Let $m$ be a bounded Borel function on $\Delta(L,-iX)$.
\begin{enumerate}
\item[\rm(i)] For $f\in \mD$,
\begin{equation*}
m(L,-iX)f=\sum_{n\in\frac12\ZZ}m(L_n,n)\mP_nf\ .
\end{equation*}
\item[\rm(ii)] The functional
\begin{equation*}
f\in \mD\ \longmapsto\ m(L,-iX)f(e)
\end{equation*}
defines a $K$-central distribution $\Phi$ on $G$ such that 
$$
m(L,-iX)f=f*\Phi\ .
$$
\item[\rm(iii)] Conversely, for any $K$-central distribution $\Phi$ such that $\|f*\Phi\|_2\le C\|f\|_2$ for every $f\in \mD$, there is a bounded Borel function $m$ on $\Delta(L,-iX)$, unique up to sets of $E$-measure zero, such that 
$$
f*\Phi=m(L,-iX)f\ .
$$
\end{enumerate}
\end{pro}

%

We shall now focus on determining the joint spectrum of the pair $(L,-iX)$ on the full space $L^p$ for $p\in (1,\infty)\setminus\{2\}.$ The key tool we use here is Theorem \ref{thm:multVnp}.

As both $iX$ and $L$ are unbouded we need first to state what do we mean by their domains. We consider $iX$ to be an operator defined on the Sobolev space in the $\theta$ variable in $x=u_{\varphi}a_tu_{\theta}.$ More precisely, let $W^{1,p}_{X}$ be the space of those functions $f$ on $G$ such that for a.e.\ $x\in G$ the function $h(\theta)=f(xu_{\theta})$ belongs to the classical Sobolev space $W^{1,p}([0,4\pi)).$ For $f\in W^{1,p}_{X}$ the quantity $iX f$ is a well defined function in $L^p.$ The space $W^{1,p}_{X}$ comes equipped with the norm
$$\|f\|_{W^{1,p}_{X}}:=\|Xf\|_p+\|f\|_p.$$ The operator $iX$ is considered on the domain  ${W^{1,p}_{X}}.$ We remark that the space  ${W^{1,p}_{X}}$ is also the domain on $L^p$ of the translation group $e^{-tX}$ acting on the $\theta$ component of $f(u_{\varphi}a_tu_{\theta})$.

Since the semigroup $\{e^{-tL}\}_{t>0}$ is strongly continuous on $L^p$ (in fact a contractive one) it is natural to consider $L$ on its domain as a generator of $\{e^{-tL}\}_{t>0}$ on $L^p.$  

Till the end of this section we fix $1<p<\infty.$ It is well known that $$\sigma_{L^p}(iX)= \ZZ/2\qquad \textrm{and}\qquad   \sigma_{L^p}(L)=\gamma\left[S_{\delta(p)}\right]=\Par(\delta(p))$$
where $\delta(p)=|1/p-1/2|$ while $$ \Par(t):=\left\{z\in \CC\colon \Real z\ge \frac{(\Ima z)^2}{4t^2}+\frac14-t^2 \right\},\qquad\textrm{if } t\neq 0 $$
and
$$\Par(0):=[1/4,\infty).$$

There are several reasonable notions of a joint spectrum for a pair of unbounded operators. These notions may but need not to coincide. We recall two of them here. 

The joint approximate spectrum $\sigma_a(L,-iX)$ is the set of all $(\xi,\la) \in \CC^2$ such that there is a sequence of vectors $f_j\in \Dom(iX)\cap \Dom(L)$ with
$$\lim_{j\to \infty} \|iX f_j-\xi f_j\|_{L^p}+\|L f_j-\la f_j\|_{L^p}=0.$$ 
The joint residual spectrum $\sigma_R(L,-iX)$ is the set of all $(\xi,\la)\in \CC$ such that $$\Ran(\xi-iX)+\Ran(\la-L) \textrm{ is not dense in } L^p.$$
The joint spectrum $\sigma_J(L,-iX)$ is $\sigma_a(L,-iX) \cup \sigma_R(L,-iX).$

The commutant spectrum $\sigma'(L,-iX)$ is the set of all pairs $(\xi,\la)$ such that the equation
\begin{equation}
\label{eq:comutant}
(\xi I-iX)B_1+(\la I - L)B_2=I\end{equation}
has no solution among operators $B_1,B_2$ belonging to the commutant $\mathcal{R}'$ of the family of operators $$\mathcal{R}:=\{(\xi-iX)^{-1}\colon \xi \in \rho(iX)\}\cup\{(\la-L)^{-1}\colon \la \in \rho(L)\} $$

To state the main result of this section we define
\begin{equation*}
\Delta^p(L,-iX):= \bigcup_{n\in \ZZ/2} \Delta^p_n\times  \{n\}\end{equation*} where
$$\Delta^p_n:=\big\{z+n^2:z\in \Par(\delta(p))\cup \gamma[D_n]\big\}.$$

\begin{thm}
	\label{thm:jointspectrum}
	For each $1<p<\infty$ we have
	\begin{equation*}
	\sigma'(L,-iX)
	=\sigma_J(L,-iX)=\Delta^p(L,-iX).
	\end{equation*}
\end{thm}
\begin{proof}
	We will apply Theorem 1 1) of Mirotin \cite{Mir2}.
	
	Note first that the pair $(L,-iX)$ satisfies condition (K) from \cite{Mir2}.
	That is, 
	\begin{enumerate}[K1]
		\item 
		if $f\in \Dom(XL)\cap \Dom L$ then $f\in \Dom(LX)$ and $XLf=LXf,$
		\item if $f\in \Dom(LX)\cap \Dom X$ then $f\in \Dom(XL)$ and $XLf=LXf,$
	\end{enumerate}
	Let us justify only K1, as the proof of K2 is similar. Write $iX$ as $|X|\chi_{iX>0}-|X|\chi_{iX\le 0}.$ Note that the projection $\chi_{iX>0}$ is bounded on all $L^p$ spaces (this is equivalent to the boundedness of the Hilbert transform). Thus, taking $g\in \Dom X= \Dom |X|$ and denoting $g_+=\chi_{iX>0}(g)$ and $g_-=\chi_{iX\le 0}(g)$ we have
	\begin{equation*}
	iX g= |X|g_+-|X|g_-=\lim_{t\to 0^+}\frac{(e^{-t|X|}-I)(g_+-g_-)}{t}.
	\end{equation*}
	Taking $g=Lf\in \Dom X,$ with $f\in \Dom L$ we have \begin{equation}
	\label{eq:Pfjointspectrum1}
	\frac{(e^{-t|X|}-I)(g_+-g_-)}{t}=L\bigg(\frac{(e^{-t|X|}-I)(f_+-f_-}{t}\bigg).\end{equation}
	Here we have also used the fact that if $f\in \Dom L$ then $e^{-t|X|}f\in \Dom L$ and $e^{-t|X|}Lf=Le^{-t|X|}f.$  Now, as $t\to 0^+$ the left hand side of \eqref{eq:Pfjointspectrum1} converges to $iX Lf$ while $\frac{(e^{-t|X|}-I)(f_+-f_-)}{t}$ goes to $iX f.$ Since $L$ is closed being a generator of the strongly continuous semigroup in $L^p$ we conclude that $Xf\in \Dom L $ and $XLf=LXf.$

	Therefore in view of the inclusion
	$$\sigma'(L,-iX)
\supseteq\sigma_J(L,-iX)$$
	proved in \cite[Theorem 1 1)]{Mir2} it is enough to show that $\sigma'(L,-iX)\subseteq \Delta^p(L,-iX)$ and that $\Delta^p(L,-iX)\subseteq \sigma_J(L,-iX).$
	
	We start with proving that $\sigma'(L,-iX)\subseteq \Delta^p(L,-iX).$ To this end we assume that $(\xi,\la)\not \in \Delta^p(L,-iX).$ If $\xi\not \in \ZZ/2$ then $B_1=(\xi I-iX)^{-1}$ and $B_2=0$ belong to $\mathcal{R}'$ and satisfy \eqref{eq:comutant}, hence $(\xi,\la)\not \in \sigma'(L,-iX)$. It remains to consider $\xi=n_0$ for some $n_0\in \ZZ/2.$ In this case we take
	$$B_1=\sum_{n\in \ZZ/2, n\neq n_0}(n_0-n)^{-1}\mP_n\qquad \textrm{and}\qquad B_2=(\la I-L)^{-1}\mP_{n_0}.$$
	
	Then $B_1$ is bounded on $L^p$  by Fourier analysis on the torus. 
	Indeed, denoting $$H(\theta)=\sum_{n\in \ZZ/2, n\neq n_0}(n_0-n)^{-1} e^{-in\theta},\qquad \theta \in \TT,$$ we have
	$$(B_1 f)(g u_{\varphi})=\fint_{\TT}f(g u_{\theta})H(\theta-\varphi)\,d\theta,\qquad g\in G, \varphi \in \TT.$$
	Moreover, it is not hard to see that
	$$H(\theta+2\pi)=\frac{\theta}{2},\qquad \theta \in (-2\pi,2\pi].$$ 
	 Therefore $H\in L^1(\TT)$ and using Cartan coordinates \eqref{eq:Cartan} together with Fubini's theorem we conclude that $B_1$ is bounded on $L^p.$
	
	  We claim that also $B_2$ is bounded on $L^p$. This follows from Theorem \ref{thm:multVnp}. Indeed, taking $m(s)=(\la -s)^{-1}$ we have
	\begin{equation*}
	m(L_{n_0})\mP_{n_0}= (\la I-L)^{-1}\mP_{n_0},\qquad\textrm{and}\qquad m_{n_0}(s)=m(n_0^2+\gamma(s))=(\la-(n_0^2+\gamma(s)))^{-1}.
	\end{equation*}
	Clearly, $m_{n_0}(s)$ extends to a bounded holomorphic function in $\Int S_{\delta(p)}.$ 
	Since $\Delta_{n_0}^p=n_0^2+\gamma\left[D_{n_0}\cup S_{\delta(p)}\right]$, we see that if $\la \not\in \Delta_{n_0}^p$ then we have  $|\la-(n_0^2+\gamma(s))|>c>0$ for $s\in D_{n_0}\cup S_{\delta(p)}.$ Then, it is straightforward to see that
	$$\sup_{s\in D_{n_0}}|m_{n_0}(s)|+\|m_{n_0}\|_{MH(S_{\delta(p)},2)}<\infty$$
	Thus, Theorem \ref{thm:multVnp} implies that $B_2= (\la I-L)^{-1}\mP_{n_0}$ is bounded on $L^p.$ This finishes the proof of the inclusion $\sigma'(L,-iX)\subseteq \Delta^p(L,-iX).$
	
	We shall now prove that $\Delta^p(L,-iX)\subseteq \sigma_J(L,-iX).$ Consider first $p>2.$ Then, by Lemma \ref{lem:Lp_disc_part} and Corollary \ref{cor:large_t} we see that the spherical function $\zns\in L^p$ for $(n,s)\in \Int \Delta^p(L,-iX).$ Thus, every such $\zns$ is a joint eigenfunction on $L^p$ for $(L,-iX)$ and $\Int \Delta^p(L,-iX)\subseteq \sigma_a(L,-iX).$ Since $\sigma_a(L,-iX)$ is closed, see \cite[Lemma 2 1)]{Mir2}, we obtain $$\Delta^p(L,-iX)\subseteq \sigma_a(L,-iX)\subseteq \sigma_J(L,-iX).$$ For $1<p<2$ we use duality. Denote by $iX_{p'}$ and $L_{p'}$ the operators $iX$ and $L$ when considered with their respective domains on  $L^{p'}(G).$
	Then, from \cite[Lemma 11]{Mir0} we have  $$\sigma_J(L,-iX)\supseteq \sigma_R(L,-iX)=\sigma_a((iX)^*,L^*)=\sigma_a(iX_{p'},L_{p'})\supseteq E_{p'}=\Delta^p(L,-iX).$$   The proof of  $\Delta^p(L,-iX)\subseteq \sigma_J(L,-iX)$ is thus completed, hence, also the proof of Theorem \ref{thm:jointspectrum}.
\end{proof}

Besides the given notions of the joint spectrum there is also the bicommutant spectrum $\sigma''(L,-iX),$ the Shilov spectrum $\sigma(L,-iX)$ and the Taylor spectrum $\sigma_T(L,-iX),$ see e.g.\cite{Mir2} for the definitions. As to the Taylor spectrum, due to the inclusions
$$
\sigma'(L,-iX)\supseteq\sigma_T(L,-iX)
\supseteq\sigma_J(L,-iX)
$$
	proved in \cite{Mir2}, we also have $\sigma_T(L,-iX)=\Delta^p(L,-iX)$.

We finish the paper by stating a holomorphic extension property of joint spectral multipliers of $(L,-iX).$ 
\begin{cor}
	\label{cor:CleSte}
	Assume that $m(L,-iX)$ is a bounded operator on $L^p$ for some $1<p<\infty,$ $p\neq 2.$ Then, for each $n\in \ZZ/2,$  the function $m(\cdot,n)$ extends to a bounded holomorphic function in $\Int (n^2+\Par(\delta(p))).$ Moreover, the bound 
	\begin{equation}
	\label{eq:CleStemult}
	\|m(\cdot,n)\|_{H^{\infty}(\Int(n^2+\Par(\delta(p)))}\le \|m(L,-iX)\|_{L^p\to L^p}
	\end{equation}
	holds uniformly in $n\in \ZZ/2.$
\end{cor}
\begin{proof}
	Using Proposition \ref{pro:CleSte} we see that for each $n\in \ZZ/2$ the function $m(n^2+\gamma(s),n)$ extends to a bounded holomorphic function on $S_{\delta(p)}.$ Since $\gamma(S_{\delta(p)})=\Par(p)$ we conclude that $m(\cdot,n)$ extends to a bounded holomorphic function on $\Int(n^2+\Par(\delta(p))).$ Finally, \eqref{eq:CleSte} implies that
	$$\|m(\cdot,n)\|_{H^{\infty}(\Int(n^2+\Par(\delta(p)))}\le \|m(L,-iX)\|_{V_n^p\to V_n^p}$$
	which leads to \eqref{eq:CleStemult}. This completes the proof of the corollary.
\end{proof}

\subsection*{Acknowledgments}  B{\l}a{\.z}ej Wr{\'o}bel thanks Wojciech M{\l}otkowski and Karol Penson for literature references. B{\l}a{\.z}ej Wr{\'o}bel was supported by
the National Science Centre, Poland (NCN) grant 2014\slash 15\slash D\slash ST1\slash 00405

\end{document}